\documentclass{amsart}
\usepackage{amsfonts,amssymb}
\usepackage{enumerate,graphicx}
\usepackage[labelfont=normal,labelformat=simple]{subcaption}
\usepackage{caption}
\usepackage[square,numbers]{natbib}
\usepackage{mathrsfs,bbm}
\usepackage{tikz,tikzscale,tikz-cd}
\usepackage{multirow,xparse,fp}
\usepackage{environ}
\usepackage{amstext}
\usepackage{hyperref}
\usepackage{diagbox}
\usepackage{float}
\usepackage{graphicx}
\usepackage{psfrag}
\usepackage{psfig}

\newtheorem{theorem}{Theorem}[section]

\newtheorem{proposition}[theorem]{Proposition}
\newtheorem{corollary}[theorem]{Corollary}
\newtheorem{lemma}[theorem]{Lemma}
\theoremstyle{definition}

\newtheorem{example}[theorem]{Example}

\newtheorem{problem}{Problem}
\newtheorem{remark}[theorem]{Remark}

\numberwithin{equation}{section}
\theoremstyle{plain} 
\newcommand{\thistheoremname}{}
\newtheorem*{genericthm*}{\thistheoremname}
\newenvironment{namedthm*}[1]
  {\renewcommand{\thistheoremname}{#1}%
   \begin{genericthm*}}
  {\end{genericthm*}}

\newcommand\Nint{\mathbb{N}}
\newcommand\alphabet{\mathcal{A}}

\newcommand\norm[1]{\lvert #1 \rvert}

\newsavebox\dotbox
\sbox{\dotbox}{\(\displaystyle\bigodot\)}
\newlength{\dotheight}
\usepackage{amsmath}
\usepackage{color}

\begin{document}

\title{On structure of topological entropy for tree-shift of finite type}

\author[J-C Ban]{Jung-Chao Ban}
\address[Jung-Chao Ban]{Department of Mathematical Sciences, National Chengchi University, Taipei 11605, Taiwan, ROC.}
\address{Math. Division, National Center for Theoretical Science, National Taiwan University, Taipei 10617, Taiwan. ROC.}
\email{jcban@nccu.edu.tw}

\author[C-H Chang]{Chih-Hung Chang}
\address[Chih-Hung Chang]{Department of Applied Mathematics, National University of Kaohsiung, Kaohsiung 81148, Taiwan, ROC.}
\email{chchang@nuk.edu.tw}

\author[W-G Hu]{Wen-Guei Hu}
\address[Wen-Guei Hu]{College of Mathematics, Sichuan University, Chengdu, 610064, China}
\email{wghu@scu.edu.cn}

\author[Y-L Wu]{Yu-Liang Wu}
\address[Yu-Liang Wu]{Department of Applied Mathematics, National Chiao Tung University, Hsinchu 30010, Taiwan, ROC.}
\email{s92077.am08g@nctu.edu.tw}

\thanks{Ban and Chang are partially supported by the Ministry of Science and Technology, ROC (Contract No MOST 109-2115-M-004-002-MY2, 109-2115-M-390 -003 -MY3 and 109-2927-I-004-501). Hu is partially supported by the National Natural Science Foundation of China (Grant No.11601355).}

\date{October 1, 2020}

\baselineskip=1.2\baselineskip

\begin{abstract}
This paper deals with the topological entropy for hom Markov shifts $\mathcal{T}_M$ on $d$-tree. If $M$ is a reducible adjacency matrix with $q$ irreducible components $M_1, \cdots, M_q$, we show that $h(\mathcal{T}_{M})=\max_{1\leq i\leq q}h(\mathcal{T}_{M_{i}})$ fails generally, and present a case study with full characterization in terms of the equality. Though that it is likely the sets $\{h(\mathcal{T}_{M}):M\text{ is binary and irreducible}\}$ and $\{h(\mathcal{T}_{X}):X\text{ is a one-sided shift}\}$ are not coincident, we show the two sets share the common closure. Despite the fact that such closure is proved to contain the interval $[d \log 2, \infty)$, numerical experiments suggest its complement contain open intervals.
\end{abstract}
\maketitle
\tableofcontents

\section{Introduction}

Entropy, which was introduced in 1865 by the German physicist and
mathematician Rudolf Clausius, plays a crucial role in different fields of
science, e.g., the information theory and ergodic theory. Such quantity
describes the complexity or richness of a dynamical system. Let $\mathcal{%
H(S)}$ be the entropies of a family of dynamical systems $\mathcal{S}$. To
examine the structure of $\mathcal{H(S)}$, one usually considers the
fundamental properties of $\mathcal{H(S)}$; namely, the monotonicity,
continuity and the denseness of $\mathcal{H(S)}$. For some classical
families of maps $\mathcal{S}$ which are defined in $\mathbb{R}$, $\mathcal{%
H(S)}$ forms the so-called \emph{Devil's staircase} function if the constant part of $\mathcal{H(S)}$ is open and dense in the
parameter space, e.g., logistic maps \cite{MT-1988}.

Before stating the main results of this work, we review some relevant results
of one-sided shifts of finite type (SFTs) $X_{M}$, where $M$ is the
associated adjacency matrix. It is known that $h(X_{M})$ is $\log \lambda
_{M}$, where $\lambda _{M}$ is the spectral radius of $M$, and that the
monotonicity follows from this property. Furthermore, the strict monotonicity holds whenever $M$ admits the irreducibility property (Theorem 4.4.7 \cite{LM-1995}). Precisely, let $M$ be irreducible, $0\leq N\leq M$ and $%
N_{k,l}<M_{k,l}$ for a pair $(k,l)$ of indices, then $h(X_{N})<h(X_{M})$.
Theorem 4.4.7 \cite{LM-1995} indicates that $h(X_{M})$ with irreducible $M$
is in the boundary of some constant part of the entropy function, that is,
every proper subshift of $X_{M}$ decreases the entropy. In addition to the monotonicity, irreducibility is a cornerstone in the theory of topological entropy. Let $M$ be a reducible matrix with irreducible components $M_1, M_2, \cdots M_q$. It is known that (Theorem 4.4.4 \cite%
{LM-1995})
\begin{equation}\label{1}
h(X_{M})=\max_{1\leq i\leq q}h(X_{M_{i}})\text{.}
\end{equation}%
This means that
\begin{equation}
\{h(X_{M}):M\text{ is binary and irreducible}\}=\mathcal{H(S)}\text{,}
\label{2}
\end{equation}%
where $\mathcal{S}$ is the family of the one-sided SFTs. 

Novel phenomena have been revealed in the tree-shifts. Tree-shifts, which was proposed by Aubrun and Beal \cite{AB-TCS2012,
AB-ToCS2013}, are shift spaces defined in free semigroups. Such shifts have
received extensive attention recently since the tree-shifts exhibit the
nature structure of one- and higher-dimensional dynamical systems and are
equipped with multiple directional shift transformations. Among all, hom tree-shifts form an important class since it was inspired by the physical
models from statistical mechanics, e.g., lattice gases and spin systems. A \emph{hom Markov
tree-shift} is a nearest neighbor shift of finite type which is isotropic,
that is, if $a,b$ are two symbols forbidden to sit next to each other in
some coordinate direction, then they are forbidden to sit next to each other
in all coordinate directions. 
Petersen and Salama introduce the \emph{topological
entropy} for tree-shifts and prove that the limit in the definition exists,
and the limit in the definition is actually the infimum \cite%
{petersen2018tree,petersen2020entropy}. For the measure-theoretic results of entropy we refer
the reader to \cite{austin2018gibbs, mairesse2017uniform} and references
therein.

Since the hom Markov tree-shift, $\mathcal{T}_{M}$, is induced from the one-dimensional shift $X_{M}$, where $M$ is
the associated adjacency matrix (see Section 2 for more details), it is
therefore of interest to study the relations between $\mathcal{T}_{M}$ and $%
X_{M}$. Petersen and Salama \cite%
{petersen2018tree} show that the topological entropy $h(\mathcal{T}_{M})$
dominates the topological entropy $h(X_{M})$. Ban et al.~\cite%
{ban2020topological} demonstrate that for irreducible $M$, $h(\mathcal{T}_{M})$ and $h(X_{M})$
are equal if and only if $M$ has the property of equal row sum (Theorem 2.1.~\cite{ban2020topological}), that is, $\underset{i}{\max}\hspace{0.1cm}\underset{j}{\sum}M_{i,j}=\underset{i}{\min}\hspace{0.1cm}\underset{j}{\sum}M_{i,j}$.
Although the relation of the $h(\mathcal{T}_{M})$ and $%
h(X_{M})$ is known, the structure of the entropy function of hom Markov
tree-shifts is far from being conclusive. For this purpose, the paper
investigates the structure of $\mathcal{H(S)}$ when $\mathcal{S}$ is the
family of hom Markov tree-shifts.

The first result of this work is to extend Theorem 4.4.7 \cite{LM-1995} to
the case where $\mathcal{S}$ is the family of hom Markov tree-shifts
(Theorem \ref{theorem:2.2}). Surprisingly, \eqref{1} is no longer true under such a circumstance. For example, suppose $b>a\geq 1$ and $0 < l\leq b$,
consider the reducible upper triangular matrix
\begin{equation*}
M(a,b;l)=%
\begin{bmatrix}
E_{a} & R_{l} \\
0 & E_{b}%
\end{bmatrix}%
\text{,}
\end{equation*}%
where $E_{a}\in \{0,1\}^{a\times a}$, $E_{b}\in \{0,1\}^{b\times b}$ are full
matrix, and $R_{l}\in \{0,1\}^{a\times b}$ has a constant row sum $l$. Theorem \ref{theorem:3.0.2} reveals that
\begin{equation}
h(\mathcal{T}_{M(a,b;l)})>\max \{h(\mathcal{T}_{E_{a}}),h(\mathcal{T}%
_{E_{b}})\}=h(\mathcal{T}_{E_{b}})=\log b\text{,}  \label{3}
\end{equation}%
for a set of $(a,b,l)\in \mathbb{N}^{3}$. The complete
characterization of $h(\mathcal{T}_{M(a,b;l)})=h(\mathcal{T}_{E_{b}})$ is
presented in Theorem \ref{theorem:3.2}. Furthermore, Theorem \ref{theorem:3.0.2-0}
extends to the case where $M$ has
$q>2$ irreducible components. We emphasize that the problem of finding all
possible values of $(a,b,l)$ in which (\ref{3}) holds for hom Markov
tree-shifts depends on the `sizes' and `forms' of the irreducible parts and
the upper right corner matrix of $M$ (Proposition \ref{proposition:5.1}). Furthermore, the structure of the entropy function of the hom Markov tree-shifts is not only determined by the family of $\mathcal{T}_{M}$ with irreducible $M$. 

Since the equality \eqref{1} is no longer true for the family of hom Markov tree-shift, the equality \eqref{2} may fail due to the fact that the collection of reducible matrices $M$ may produce more values of entropy. However, Theorem \ref{thm:denseness_irreducible_shift} together with Corollary \ref{cor:universal_closure} show that 
\begin{equation}
    \overline{\{h(\mathcal{T}_M): M \text{ is binary and irreducible}\}} = \overline{\{h(\mathcal{T}_X):X \text{ is a shift space}\}} = \overline{\mathcal{H(S)}},
\end{equation}
where $\mathcal{S}$ is the set of hom Markov tree-shift and $\overline{E}$ is the closure of the set $E$. The novel phenomenon indicates that the set of hom Markov tree-shifts $\mathcal{T}_M$ with irreducible $M$ is the set of basic `building blocks' of $\mathcal{S}$ for the entropy.

Let $\mathcal{S}$ be the family of the one-sided SFTs. As we mentioned
before, $h(X_{M})=\log \lambda _{M}$, where $\lambda _{M}$ is the spectral
radius of $M$. The converse is also true if $M$ is an aperiodic matrix. To
be precise, if $\lambda $ is a Perron number, then there is a nonnegative
aperiodic integral matrix whose spectral radius is $\lambda $ \cite%
{lind1983entropies, Lind-ETaDS1984}. Combining this with the fact that the set
of Perron value is dense in $[1,\infty )$ we conclude that $\mathcal{H(S)}$
is dense in $[0,\infty )$.
However, this is not generally true if $\mathcal{S%
}$ is the family of hom Markov tree-shifts.
Only the case for $2$-tree is discussed below, and the cases for general $d$-tree have a similar phenomenon.
Ban et al.~prove that $h(%
\mathcal{T}_{M})=0$ or $h(\mathcal{T}_{M})\geq \frac{1}{2}\log 2$
(Proposition 4.1 \cite{ban2020topological}), that is, $\mathcal{H(S)}$ is
not dense in $(0,\frac{1}{2}\log 2)$. The question whether $\mathcal{H(S)}$ is
dense in $[\frac{1}{2}\log 2,\infty )$ arises naturally. In Theorem \ref{theorem:2.2.1}, we
prove that $\{h(\mathcal{T}_{M}):M$ is binary and irreducible$\}$ is dense
in $[\log 2,\infty )$. It is worth pointing out that the numerical results
suggest that there exist intervals in $[\frac{1}{2}\log 2$, $\log 2)$ in which $\mathcal{H(S)}$ is not dense. We also note that the $\mathcal{T}%
_{M}$ we construct in Theorem \ref{theorem:2.2.1} has the property that $M$ is irreducible.
However, the results in Section 2 reveal that the $h(\mathcal{T}_{M})$ with
reducible $M$ may increase the values of $\underset{1\leq i\leq q}{\max}h(\mathcal{T}%
_{M_{i}})$, where $M_{i}$ is the irreducible component in the diagonal part
of $M$, i.e., those $\mathcal{T}_{M}$ may increase the possible values of $\{h(%
\mathcal{T}_{M}):M$ is binary and irreducible$\}$. Thus we suspect that
there is another way to prove more intervals less than $\log 2$ in which $%
\mathcal{H(S)}$ is dense. The main difficulty occurs in the study of the
entropy structure of $h(\mathcal{T}_{M})$, where $M$ is a reducible matrix.
This question is at present far from being solved.

\section{Preliminary}

\subsection{Notations and definitions}

First, some necessary notations and definitions are introduced. For $d\geq 2$, denote by $\Sigma=\{0,1,\cdots,d-1\}$. Let the \emph{$d$-tree}
$\Sigma^{*}=\cup_{n\geq0}\Sigma^{n}$ be the set of all finite words on $\Sigma$, where $\Sigma^{n}$ is the set all words with length $n\geq 1$ and    $\Sigma^{0}=\{\epsilon\}$ consists of the empty word $\epsilon$. A \emph{labeled tree} $t:\Sigma^{*}\rightarrow\mathcal{A}$ is a global configuration on $d$-tree with finite alphabet $\mathcal{A}$.
Given a labeled tree and a node $w\in\Sigma^{*}$, $t_{w}$ is the label of $t$ on $w$.
 Denote by $\Delta_{n}=\cup_{i=0}^{n}\Sigma^{i}$ the $n$-th subtree of $\Sigma^{*}$. An \emph{$n$-block} $u$ is defined by $u=t|_{\Delta_{n}}$ for some labeled tree $t$.

For a one-dimensional shift space $X=X_{\mathcal{F}}\subseteq\mathcal{A}^{\mathbb{N}}$ with forbidden set $\mathcal{F}$, denote by $\mathcal{B}_{n}(X)$ the set of all feasible $n$-words of $X$. Given a shift space $X$,
the associated \emph{hom tree-shift} $\mathcal{T}_{X}\subseteq \mathcal{A}^{\Sigma^{*}}$ is the set of all labeled trees $t$ such that
$\{t_{w_{i}}\}_{i\geq0}=(t_{w_{0}},t_{w_{1}},\cdots)\in X$ for any node $w_{i}\in\Sigma^{i}$, $i\geq 0$. In particular, given a binary matrix $M=[M_{i,j}]$ indexed by $\mathcal{A}$, a \emph{Markov shift} $X_{M}$ is defined by
\begin{equation*}
X_{M}=\left\{ x=\{x_{i}\}\in\mathcal{A}^{\mathbb{N}}:M_{x_{i},x_{i+1}}=1 \text{ for }i\geq0 \right\},
\end{equation*}
and the associated hom tree-shift $\mathcal{T}_{M}=\mathcal{T}_{X_{M}}$ is called a \emph{hom Markov tree-shift}.

Given a hom Markov tree-shift $\mathcal{T}_{M}$, we denote by $\mathcal{B}_n(\mathcal{T}_{M})$ the set of $n$-blocks and by $p(n)=p_{M}(n)$ the number of $n$-blocks of $\mathcal{T}_{M}$ for $n\geq 0$. Furthermore, if the graph representation $G$ of $\mathcal{T}_M$ is specified, we write $p_{G}(n)=p_{M}(n)$. The \emph{topological entropy} $h(\mathcal{T}_{M})$ of $\mathcal{T}_{M}$ is defined by
\begin{equation*}
h(\mathcal{T}_{M})=\underset{n\rightarrow\infty}{\lim}\frac{\log p(n)}{|\Delta_{n}|},
\end{equation*}
which measures the growth rate of the number of $n$-blocks of $\mathcal{T}_{M}$. The existence of the limit is proved by Petersen and Salama \cite{petersen2018tree,petersen2020entropy}.
Assume $M\in\{0,1\}^{k\times k}$, $k\geq 1$. For $1\leq i\leq k$, denote by $p_{M;i}(n)$ (or $p_{i}(n)$ if there is no confusion) the number of  $n$-blocks $u$ for $\mathcal{T}_{M}$ with $u_{\epsilon}=i$. For a finite sequence $\mathbf{s}_{n}=\left\{s_{i}\right\}_{i=1}^{d^{n}}$ where $1\leq s_{i}\leq k$, $1\leq i\leq d^{n}$, denote by $p_{M;\mathbf{s}_{n}}(n)$ the number of $n$-blocks for $\mathcal{T}_{M}$ whose labels on the bottom layer $\Sigma^{n}$ from left to right are $s_{1}, s_{2}, \cdots, s_{d^{n}}$. If $G$ is the graph representation of $\mathcal{T}_M$, we also denote $p_{G;\mathbf{s}_{n}}(n)=p_{M;\mathbf{s}_{n}}(n)$.

Petersen and Salama \cite{petersen2018tree} have demonstrated that if $M$ is irreducible,
\begin{equation}\label{eqn:2.0.1}
\underset{n\rightarrow\infty}{\limsup}\frac{\log p_{i}(n)}{|\Delta_{n}|}=h(\mathcal{T}_{M})
\end{equation}
for $1\leq i\leq k$;
furthermore, if $M$ is primitive,
\begin{equation*}
\underset{n\rightarrow\infty}{\lim}\frac{\log p_{i}(n)}{|\Delta_{n}|}=h(\mathcal{T}_{M})
\end{equation*}
for $1\leq i\leq k$.

\subsection{Monotonicity of $h(\mathcal{T}_{M})$ for irreducible $M$ }

In this subsection, the strict monotonicity of $h(\mathcal{T}_{M})$ for irreducible $M$ is considered. More specifically, we will show that if $B$ is irreducible and $A=[a_{i,j}]_{k\times k}>B=[b_{i,j}]_{k\times k}$, then $h(\mathcal{T}_{A})>h(\mathcal{T}_{B})$, where $A>B$ means $a_{i,j}\geq b_{i,j}$ for all $1\leq i,j\leq k$ and there exist $1\leq i',j'\leq k$ such that $a_{i',j'}> b_{i',j'}$. Before showing the main result, we prove the following lemma.

\begin{lemma}
\label{lemma:2.1}
Suppose $A$ and $B\in\{0,1\}^{k\times k}$. If $B$ is irreducible and $A>B$, then there exists $N\geq1$ such that
\begin{equation}\label{eqn:2.1}
 p_{A;i}(N)>p_{B;i}(N)
\end{equation}
for $1\leq i\leq k$.
\end{lemma}

\begin{proof}
Let $A=[a_{i,j}]$ and $B=[b_{i,j}]$. Since $A>B$, without loss of generality, we assume $a_{i^{\ast},j^{\ast}}=1$ and $b_{i^{\ast},j^{\ast}}=0$ for some $1\leq i^{\ast},j^{\ast}\leq k$.
Because $B$ is irreducible, for $1\leq i\leq k$, we have $\sum_{j=1}^{k}(B^{m})_{i,j}\geq 1$ for all $m\geq 1$, and there exists $n(i,i^{\ast})\geq 1$ such that $\left(B^{n(i,i^{\ast})}\right)_{i,i^{\ast}}\geq 1$.
Clearly,
\begin{equation*}
\begin{array}{rcl}
\left(B^{n(i,i^{\ast})}\right)_{i,i^{\ast}}\cdot b_{i^{*},j^{*}} = 0  & \text{and}
 &\left(A^{n(i,i^{\ast})}\right)_{i,i^{\ast}}\cdot a_{i^{*},j^{*}} \geq 1. \\
\end{array}
\end{equation*}
Then, it can be proven that
\begin{equation*}
\begin{array}{rcl}
\sum_{j=1}^{k}\left(A^{n(i,i^{\ast})+m}\right)_{i,j} & > & \sum_{j=1}^{k}\left(B^{n(i,i^{\ast})+m}\right)_{i,j}.\\
\end{array}
\end{equation*}
for all $m\geq 1$. By taking $N=\underset{1\leq i\leq k}{\max}\left\{n(i,i^{\ast})+1\right\}$,
\begin{equation*}
\begin{array}{rcl}
\sum_{j=1}^{k}\left(A^{N}\right)_{i,j} & > & \sum_{j=1}^{k}\left(B^{N}\right)_{i,j}\\
\end{array}
\end{equation*}
for all $1\leq i \leq k$, which implies
\begin{equation}\label{eqn:2.2}
\begin{array}{rl}
 & \left|\left\{ x=(x_{0},x_{1},\cdots, x_{N})\in \mathcal{B}_{N+1}(X_{A}):x_{0}=i \right\}\right| \\
 >  & \left|\left\{ x=(x_{0},x_{1},\cdots, x_{N})\in \mathcal{B}_{N+1}(X_{B}):x_{0}=i \right\}\right|
\end{array}
\end{equation}
for all $1\leq i \leq k$.

From the irreducibility of $A$, for any feasible $(n+1)$-word $x=(x_{0},x_{1},\cdots, x_{n})\in \mathcal{B}_{n+1}(X_{A})$,
\begin{equation*}
\left\{u\in \mathcal{T}_{A}\mid_{\Delta_{n}}: u_{\epsilon}=x_{0} \text{ and } u_{0^{j}}=x_{j}, 1\leq j\leq n  \right\}\neq\emptyset.
\end{equation*}
The case for $B$ is also valid.
Therefore, by $A>B$ and (\ref{eqn:2.2}), (\ref{eqn:2.1}) follows immediately. The proof is complete.
\end{proof}

The following theorem shows that the strict monotonicity of $h(\mathcal{T}_{M})$ for irreducible $M$ holds.
\begin{theorem}
\label{theorem:2.2}
Let $A$ and $B\in\{0,1\}^{k\times k}$. If $B$ is irreducible and $A>B$, then $h(\mathcal{T}_{A})>h(\mathcal{T}_{B})$.
\end{theorem}

\begin{proof}
From Lemma \ref{lemma:2.1}, there exists $N\geq 1$ such that $p_{A;i}(N)>p_{B;i}(N)$ for all $1\leq i\leq k$. Let $c=\underset{1\leq i\leq k}{\min}\frac{p_{A;i}(N)}{p_{B;i}(N)}>1$.
For $n\geq 1$,
\begin{equation*}
\begin{array}{rcl}
p_{A}(n+N) & = & \underset{1\leq s_{1},s_{2},\cdots,s_{d^{n}}\leq k}{\sum}p_{A;\mathbf{s}_{n}}(n)\cdot \left(\prod_{j=1}^{d^{n}}p_{A;s_{j}}(N)\right)\\
& & \\
 & \geq  &  \underset{1\leq s_{1},s_{2},\cdots,s_{d^{n}}\leq k}{\sum}p_{B;\mathbf{s}_{n}}(n)\cdot \left(\prod_{j=1}^{d^{n}}p_{A;s_{j}}(N)\right)\\
& & \\
 & \geq  &  \underset{1\leq s_{1},s_{2},\cdots,s_{d^{n}}\leq k}{\sum}p_{B;\mathbf{s}_{n}}(n)\cdot c^{d^{n}}\cdot \left(\prod_{j=1}^{d^{n}}p_{B;s_{j}}(N)\right)\\
& & \\
  & = & c^{d^{n}}\cdot p_{B}(n+N).
\end{array}
\end{equation*}
Therefore,
\begin{equation*}
\begin{array}{rcl}
h(\mathcal{T}_{A}) & = & \underset{n\rightarrow\infty}{\limsup}\frac{\log p_{A}(n+N)}{\left|\Delta_{n+N}\right|}\\
 & & \\
 & \geq &  \underset{n\rightarrow\infty}{\limsup}\frac{\log \left(c^{d^{n}}\cdot p_{B}(n+N)\right)}{\left|\Delta_{n+N}\right|} \\
 & & \\
 & = & \frac{(d-1)\log c}{d^{N+1}}+h(\mathcal{T}_{B})\\
 & & \\
 & > & h(\mathcal{T}_{B}).
\end{array}
\end{equation*}
The proof is complete.
\end{proof}

\begin{remark}
\label{remarl:2.1.10}
Theorem \ref{theorem:2.2} is analogous to Theorem 4.4.7 \cite{Lind-ETaDS1984} for one-dimensional irreducible SFTs. We also emphasize that, as one-dimensional reducible SFTs, the strict monotonicity of $h(\mathcal{T}_{M})$ for reducible $M$ is not true, that is, if $B$ is reducible and $A>B$, it could happen that $h(\mathcal{T}_{A})=h(\mathcal{T}_{B})$. For example, let $A=\left[\begin{array}{ccc} 1 & 1 & 0 \\ 0 & 1 & 1\\ 0 & 1 & 1\end{array}\right]$ and $B=\left[\begin{array}{ccc} 1 & 0 & 0 \\ 0 & 1 & 1\\ 0 & 1 & 1\end{array}\right]$. From Theorem \ref{theorem:3.2}, it can be checked immediately that $h(\mathcal{T}_{A})=h(\mathcal{T}_{B})=\log2$.
\end{remark}

\section{$h(\mathcal{T}_{M})$ for reducible $M$}
Let the irreducible components of reducible $M$ are $M_{1},M_{2},\cdots,M_{q}$. In one-dimensional reducible shifts of finite type, it is known that $h(X_{M})=\underset{1\leq i\leq q}{\max}\{h(X_{M_{i}})\}$ and $h(X_{M_{i}})$ is equal to the logarithm of the spectral radius of $M_{i}$. However,
the preceding equality is no longer true for the hom Markov tree-shifts. Our results below demonstrate rather strikingly that the criterion of the equality depends on the set of all irreducible components of $M$ and the connections among them.
In this section, we consider the case that all irreducible components of $M$ are of the form $M_{i}=E_{n_{i}}$, where $E_{n}$ is the $n\times n$ matrix whose all entries are $1$. Clearly, $\underset{1\leq i\leq q}{\max}\{h(\mathcal{T}_{M_{i}})\}=\underset{1\leq i\leq q}{\max}\{n_{i}\}$.

\subsection{Reducible $M$ with two irreducible components}
This subsection considers a certain class of the reducible $M=M(a,b;l)$ with exactly two irreducible components $E_{a}$ and $E_{b}$. Our goal is to provide checkable necessary and sufficient conditions to determine whether $h(\mathcal{T}_{M})=\max\left\{h(\mathcal{T}_{E_a}),h(\mathcal{T}_{E_b})\right\}$. 

For any $a,b\geq 1$ and $0\leq l\leq b$, consider the reducible matrix

\begin{equation}\label{eqn:3.1}
M=M(a,b;l)=
\left[
\begin{array}{cc}
E_{a} & R_{l} \\
O & E_{b}
\end{array}
\right],
\end{equation}
where $O$ is $b\times a$ zero matrix and $R_{l}\in\{0,1\}^{a\times b}$ has the same row sum $l$. Clearly, $E_{a}$ and $E_{b}$ are the irreducible components of $M$. Notably, for $a\geq b$, the relationship between $h(\mathcal{T}_{M(a,b;l)})$ and $\max\{h(\mathcal{T}_{E_{a}}),h(\mathcal{T}_{E_{b}})\}=\log a$ is straightforward as below. If $a\geq b$ and $l=0$, it is clear that $h(\mathcal{T}_{M})=\log a$; if $a\geq b$ and $1\leq l\leq b$, we can check that $p(n)\geq a^{|\Delta_{n-1}|}\cdot (a+l)^{d^{n}}$, and then $h(\mathcal{T}_{M})>\log a$. Therefore, we always assume that $b>a$. The following lemma is crucial to obtain the main results.

\begin{lemma}
\label{lemma:3.0.1}
 Suppose $\alpha>0$, $\beta\geq 0$ and $d\geq 2$. Let $f(x)=\left(\alpha x+\beta\right)^{d}$ and  $\chi_{n+1}=f(\chi_{n})$, $n\geq 0$, with $\chi_{0}=1$. Let the largest real root and second largest real root of $f(x)=x$ on $[1,\infty)$ be $x_{+}$ and $x_{-}$ respectively (if $x_{+}$ exists but $x_{-}$ does not exist, let $x_{-}=x_{+}$). Then, for $\alpha+\beta>1$,
$\{\chi_{n}\}$ is increasing, and
\begin{itemize}
\item[(a)]if $x_{+}$ does not exist or $x_{+}<1$, then $\{\chi_{n}\}$ approaches $\infty$ as $n$ tends to $\infty$,
\item[(b)]if $x_{+}\geq 1$, then $\{\chi_{n}\}$ approaches $x_{-}$ as $n$ tends to $\infty$.

\end{itemize}
\end{lemma}

\begin{proof}
First, we have $f'(x)>0$ for $x\geq 1$. Since $\alpha+\beta>1$, we have $\chi_{1}>\chi_{0}=1$ and then $\{\chi_{n}\}$ is increasing.
 For $x_{+}$ does not exist or $x_{+}<1$, we have $f(x)>x$ for $x\geq 1$ and then $f'(x)>1$ for $x\geq 1$, which implies $\chi_{n+2}-\chi_{n+1}=f(\chi_{n+1})-f(\chi_{n})>\chi_{n+1}-\chi_{n}$ for $n\geq 0$. Hence,
 $\chi_{n}\rightarrow\infty$ as $n\rightarrow\infty$.

Let $g(x)=f(x)-x$. It can be verified that there exists at most one real root of $g'(x)=0$ on $[1,\infty)$, and then there exist at most two real roots of $g(x)=0$ on $[1,\infty)$. When $x_{+}\geq1$, $x_{-}$ is the real root of $f(x)=x$ that is closest to $1$ on $[1,\infty)$. Since $\chi_{n}$ is increasing, it can be shown that $\chi_{n}$ approach $x_{-}$ as $n\rightarrow\infty$. The proof is complete.

\end{proof}

Now, for $b>a$, the following theorem provides a complete classification of whether $h(\mathcal{T}_{M(a,b;l)})=\log b$.
\begin{theorem}
\label{theorem:3.0.2}
Suppose $M=M(a,b;l)$ is defined as (\ref{eqn:3.1}) with $b>a$  and $\mathcal{T}_{M}$ is defined on $d$-tree, $d\geq 2$. Let the maximal real root of $f(x)\equiv\left(\frac{a}{b}x+\frac{l}{b}\right)^{d}=x$ be $x_{+}$.
\begin{itemize}
\item[(a)] When $a+l\leq b$, $h(\mathcal{T}_{M})=\log b$.
\item[(b)] When $a+l>b$,
\begin{itemize}
\item[(i)]if $x_{+}$ does not exist or $x_{+}<1$, then $h(\mathcal{T}_{M})>\log b$,
\item[(ii)]if $x_{+}\geq 1$, then  $h(\mathcal{T}_{M})=\log b$.
\end{itemize}
\end{itemize}
\end{theorem}

\begin{proof}
First, if $a+l\leq b$, the estimation $b^{|\Delta_{n}|}\leq p(n)\leq (a+b)\cdot b^{|\Delta_{n}|}$ can be verified. Then, $h(\mathcal{T}_{M})=\log b$ follows immediately.

For $a+l> b$, it follows from \cite{petersen2018tree} that for $1\leq i\leq a+b$ and $n\geq 0$, we have

\begin{equation*}
p_{i}(n+1)=\footnotesize{\left(M\mathbf{p}(n)\right)_{i}^{d}}
\end{equation*}
where $\mathbf{p}(n)=\left[p_{1}(n), p_{2}(n), \cdots, p_{a+b}(n)\right]^{T}$.
By the definition of $M(a,b;l)$, $p_{1}(n)=p_{2}(n)=\cdots =p_{a}(n)$ and $p_{a+1}(n)=p_{a+2}(n)=\cdots =p_{a+b}(n)$ for $n\geq 1$. Then, it can be verified that
\begin{equation*}
 p_{i}(n+1)=
 \left\{
 \begin{array}{rl}
  \left(a p_{a}(n)+lp_{a+b}(n)\right)^{d} & \text{ for } 1\leq i\leq a, \\
   & \\
    \left(b p_{a+b}(n)\right)^{d}=b^{|\Delta_{n+1}|-1} & \text{ for } a+1\leq i\leq a+b,
 \end{array}
 \right.
\end{equation*}
with $p_{i}(0)=1$ for all $1\leq i\leq a+b$. Also,
\begin{equation*}
p(n) = ap_{a}(n)+b p_{a+b}(n)=ap_{a}(n)+b^{|\Delta_{n}|}.
\end{equation*}
Hence,
\begin{equation}\label{eqn:2.3.2}
h(\mathcal{T}_{M})=\max\left\{\underset{m\rightarrow\infty}{\limsup}\frac{\log p_{a}(m)}{|\Delta_{m}|}, \log b\right\}.
\end{equation}

Let $\chi_{n}\equiv p_{a}(n)/p_{a+b}(n)=p_{a}(n)/b^{|\Delta_{n}|-1}$ for $n\geq 0$. Thus, we have
\begin{equation*}
\chi_{n+1}=\left(\frac{a}{b}\chi_{n}+\frac{l}{b}\right)^{2}
\end{equation*}
with $\chi_{0}=1$. First, if $\{\chi_{n}\}$ is bounded, it can be proven that
\begin{equation*}
\underset{m\rightarrow\infty}{\limsup}\frac{\log p_{a}(m)}{|\Delta_{m}|}=\underset{m\rightarrow\infty}{\limsup}\frac{\log p_{a+b}(m)+\log \chi_{m}}{|\Delta_{m}|}= \log b.
\end{equation*}
Then, $h(\mathcal{T}_{M})=\log b$.

Secondly, we aim to show that if $\chi_{n}\rightarrow\infty$ as $n\rightarrow\infty$, then $h(\mathcal{T}_{M})>\log b$.
It is clear that $\chi_{n+1}\geq \left(\frac{a}{b}\chi_{n}\right)^{d}$, which implies that for $m\geq 1$,
\begin{equation*}
\chi_{n+m}\geq (a/b)^{(d^{m+1}-d)/(d-1)}\cdot \chi_{n}^{d^{m}}.
\end{equation*}
Then, for $n\geq 1$, it can be verified that
\begin{equation}\label{eqn:2.3.3-0}
\begin{array}{rcl}
\underset{m\rightarrow\infty}{\limsup}\hspace{0.2cm}\frac{\log \chi_{n+m}}{|\Delta_{n+m}|}
& \geq d^{-n-1} ((d-1)\log \chi_{n}-d\log(b/a)).
\end{array}
\end{equation}
Therefore, for $n\geq 1$,
\begin{equation*}
\begin{array}{rcl}
\underset{m\rightarrow\infty}{\limsup}\frac{\log p_{a}(m)}{|\Delta_{m}|} & = & \underset{m\rightarrow\infty}{\limsup}\frac{\log p_{a+b}( m)\cdot\chi_{m}}{|\Delta_{m}|}\\
& & \\
& = & \log b + \underset{m\rightarrow\infty}{\limsup}\frac{\log \chi_{m}}{|\Delta_{m}|}\\
& & \\
& = & \log b + \underset{m\rightarrow\infty}{\limsup}\frac{\log \chi_{n+m}}{|\Delta_{n+m}|}\\
 & & \\
 & \geq & \log b + d^{-n-1} ((d-1)\log \chi_{n}-d\log(b/a)),
\end{array}
\end{equation*}
which yields if $\chi_{n}\rightarrow\infty$ as $n\rightarrow \infty$, then
\begin{equation*}
h(\mathcal{T}_{M})=\underset{m\rightarrow\infty}{\limsup}\frac{\log p_{a}(m)}{|\Delta_{m}|}>\log b.
\end{equation*}
Applying Lemma \ref{lemma:3.0.1} yields Case(b), which completes the proof.

\end{proof}
Notably, by (\ref{eqn:2.3.2}) and (\ref{eqn:2.3.3-0}), if there exist $N\geq 1$ such that $\chi_{N}>(b/a)^{d/(d-1)}$, then $h(\mathcal{T}_{M})>\log b$; otherwise, $h(\mathcal{T}_{M})=\log b$. Table 1 provides numerical results for each case in Theorem \ref{theorem:3.0.2}.

\begin{center}
\begin{tabular}{|c|c|c|c|c|c|c|c|}
\hline
 Case in Thm \ref{theorem:3.0.2} & d &  a & b & l & $x_{+}$ & $h(\mathcal{T}_{M})$ & $\log b$\\
\hline
(a) & 3 & 2 & 3 & 1 & 1 & 0.477121 & 0.477121 \\
\hline
(b)(i)  & 3 & 2 & 3 & 2 & -3.21353 & 0.538423 & 0.477121 \\
\hline
(b)(ii) & 3 & 2 & 4 & 1 & 2.03407 & 0.60206& 0.60206
 \\
\hline
\end{tabular}
\end{center}
\begin{equation*}
\text{Table 1. }
\end{equation*}

Table 2 is a list of $M=M(1,5;5)$ whose $h(\mathcal{T}_{M})=\log 5$ for $d=2$ and $h(\mathcal{T}_{M})>\log 5$ for $3\leq d\leq 5$ numerically.
\begin{center}
\begin{tabular}{|c|c|c|c|c|c|c|c|}
\hline
 Case in Thm \ref{theorem:3.0.2} & d &  a & b & l & $x_{+}$ & $h(\mathcal{T}_{M})$ & $\log b$\\
\hline
(b)(ii) & 2 & 1 & 5 & 5 & 13.09017 & 0.69897 & 0.69897  \\
\hline
(b)(i)  & 3 & 1 & 5 & 5 & -18.136825 & 0.699188 & 0.69897 \\
\hline
(b)(i) & 4 & 1 & 5 & 5 & does not exist & 0.702073 & 0.69897 \\
\hline
(b)(i) & 5 & 1 & 5 & 5 & -13.40247 & 0.706288 & 0.69897 \\
\hline
\end{tabular}
\end{center}
\begin{equation*}
\text{Table 2.}
\end{equation*}
Notably, we conjecture that for a given binary matrix $M$, $h(\mathcal{T}_{M})$ increases with $d$. Up to now, it is still open.

In the following, a more detailed discussion for the case $d=2$ of Theorem \ref{theorem:3.0.2} is provided.

\begin{theorem}
\label{theorem:3.2}
Suppose $M=M(a,b;l)$ is defined as (\ref{eqn:3.1}) with $b>a$ and $\mathcal{T}_{M}$ is defined on $2$-tree. When $a+l\leq b$, $h(\mathcal{T}_{M})=\log b$. When $a+l>b$,
\begin{itemize}
\item[(a)]if $b^{2}-4al<0$, then $h(\mathcal{T}_{M})>\log b$,
\item[(b)]if $b^{2}-4al>0$
 \begin{itemize}
 \item[(i)] and $b^{2}-2al-2a^{2}<0$, then $h(\mathcal{T}_{M})>\log b$,
 \item[(ii)] and $b^{2}-2al-2a^{2}\geq0$, then $h(\mathcal{T}_{M})=\log b$,
 \end{itemize}
\item[(c)]if $b^{2}-4al=0$
\begin{itemize}
 \item[(i)] and $a\leq l$, then $h(\mathcal{T}_{M})=\log b$,
 \item[(ii)] and $a>l$, $h(\mathcal{T}_{M})>\log b$.
 \end{itemize}
\end{itemize}
\end{theorem}

\begin{proof}
Let $f(x)=\left(\frac{a}{b}x+\frac{l}{b}\right)^{2}$.
By solving $f(x)=x$, we have
\begin{equation*}
x=\frac{b^{2}-2al\pm b\sqrt{b^{2}-4al}}{2a^{2}}.
\end{equation*}
In Case (a), it is clear that $x_{+}$ does not exist.
In Case (b), there are exactly two real roots of $f(x)=x$. By making further discussing, if $b^{2}-2al-2a^{2}<0$, $x=1$ is on the right hand side of both real roots; if $b^{2}-2al-2a^{2}=0$,  $x=1$ is between the two real roots; if $b^{2}-2al-2a^{2}>0$, $x=1$ is on the left hand side of both two real roots. In Case (c), the only real root of $f(x)=x$ is $x=x_{+}=l/a$. We have that when $a\leq l$, $x_{+}\geq 1$, and when $a>l$, $x_{+}<1$.

Therefore, the results in this theorem follow by Theorem \ref{theorem:3.0.2}. The proof is complete.
\end{proof}

Table 3 provides numerical results of each case in Theorem \ref{theorem:3.2}.
\begin{center}
\begin{tabular}{|c|c|c|c|c|c|}
\hline
 Case in Thm \ref{theorem:3.2} &  a & b & l & $h(\mathcal{T}_{M})$ & $\log b$\\
\hline
(a) & 2 & 3 & 2 & 0.517166 & 0.477121 \\
\hline
(b)(i) & 2 & 3 & 1 & 0.477121&  0.477121 \\
\hline
(b)(ii) & 2 & 4 & 1 & 0.60206 &  0.60206\\
\hline
(c)(i) & 2 & 4 & 2  & 0.60206 & 0.60206 \\
\hline
(c)(ii) & 9 & 12 & 4 & 1.099264 & 1.079181  \\
\hline
\end{tabular}
\end{center}
\begin{equation*}
\text{Table 3.}
\end{equation*}

\subsection{Reducible $M$ with more irreducible components}

 In this section, we extend the results of reducible matrix $M$ with two irreducible components to general cases, that is, the reducible matrix $M$ has $q\geq 2$ irreducible components. First, we set up notation and terminology.
For $q\geq 2$, let $\mathbf{a}_{q}=(a_{1},a_{2},\cdots,a_{q})$ with $a_{q}> a_{i}\geq1$ for $1\leq i\leq q-1$, and $\mathbf{l}_{q}=\{l_{i,j}:i+1\leq j\leq q,1\leq i\leq q-1 \}$ with $0\leq l_{i,j}\leq a_{j}$. We consider the reducible matrix

\begin{equation}\label{eqn:2.2.1}
M=M(\mathbf{a}_{q};\mathbf{l}_{q})=
\left[
\begin{array}{ccccc}
E_{a_{1}} & R_{l_{1,2}} &  R_{l_{1,3}}  &  \cdots &  R_{l_{1,q}}\\
 & & & & \\
O & E_{a_{2}} &  R_{l_{2,3}}  &  \cdots &  R_{l_{2,q}}\\
 & & & & \\
O & O & E_{a_{3}} &  \cdots &  R_{l_{3,q}}\\
\vdots & \vdots & \vdots & \ddots & \vdots \\
O & O & O&  \cdots &  E_{a_{q}}
\end{array}
\right]
\end{equation}
with $q$ irreducible components and assume $\mathcal{T}_{M}$ is hom Markov tree-shift on $d$-tree, in which $R_{l_{i,j}}$ is a matrix all of whose row sums are $l_{i,j}$.
Let $b_{0}=0$ and $b_{j}=\Sigma_{i=1}^{j}a_{i}$, $1\leq j\leq q$. From \cite{petersen2018tree}, we have $p_{i}(0)=1$, $1\leq i\leq b_{q}$, and for $n\geq 1$, $p_{b_{i}+1}(n)=p_{b_{i}+2}(n)=\cdots=p_{b_{i+1}}(n)$, $0 \leq i\leq q-1$. More precisely, for $1 \leq i\leq q$,
\begin{equation}\label{eqn:2.2.2}
p_{b_{i}}(n+1)=\left(a_{i}p_{b_{i}}(n)+\overset{q}{\underset{j=i+1}{\sum}} l_{i,j}\cdot p_{b_{j}}(n)\right)^{d},
\end{equation}
for $n\geq 0$. In particular, $p_{b_{q}}(n)=a_{q}^{|\Delta_{n}|-1}$ for $n\geq 0$.
Therefore,
\begin{equation}\label{eqn:2.2.3}
h(\mathcal{T}_{M})=\underset{1\leq i\leq q}{\max}\left\{\underset{m\rightarrow\infty}{\limsup}\frac{\log p_{b_{i}}(m)}{|\Delta_{m}|}\right\},
\end{equation}
where
\begin{equation*}
\underset{m\rightarrow\infty}{\limsup}\frac{\log p_{b_{q}}(m)}{|\Delta_{m}|}=\log a_{q}.
\end{equation*}

Given $1 \leq i\leq q-1$, let
\begin{equation*}
\chi^{(i)}_{n}\equiv \frac{ p_{b_{i}}(n)}{p_{b_{q}}(n)}= \frac{ p_{b_{i}}(n)}{a_{q}^{|\Delta_{n}|-1}},
\end{equation*}
for $n\geq 0$. Clearly, $\chi^{(i)}_{0}=1$ for $1 \leq i\leq q-1$. It can be verified that for $1\leq i\leq q-1$,
\begin{equation}\label{eqn:2.2.4-0}
\chi^{(i)}_{n+1}=\frac{1}{a_{q}^{d}}\left(a_{i}\chi^{(i)}_{n}+\overset{q}{\underset{j=i+1}{\sum}} l_{i,j}\chi_{n}^{(j)}\right)^{d},
\end{equation}
for $n\geq 0$.

In the following theorem, a necessary and sufficient condition for determining whether $h(\mathcal{T}_{M(\mathbf{a}_{q};\mathbf{l}_{q})})=\log a_{q}$ is provided.
\begin{theorem}
\label{theorem:3.0.2-0}
Suppose $M=M(\mathbf{a}_{q};\mathbf{l}_{q})$, $q\geq 2$, is defined as (\ref{eqn:2.2.1}) and $\mathcal{T}_{M}$ is defined on $d$-tree. If there exist $1\leq i\leq q-1$ and $n\geq 1$ such that $\chi^{(i)}_{n}>(a_{q}/a_{i})^{d/(d-1)}$, then $h(\mathcal{T}_{M})>\log a_{q}$; otherwise, $h(\mathcal{T}_{M})=\log a_{q}$.
\end{theorem}

\begin{proof}
By (\ref{eqn:2.2.4-0}), for $1\leq i\leq q$, it is seen that $\chi^{(i)}_{n+1}\geq\left(\frac{a_{i}}{a_{q}}\chi_{n}\right)^{d}$, $n\geq 0$.
As (\ref{eqn:2.3.3-0}), it can be established that
\begin{equation}\label{eqn:2.2.4-1}
\begin{array}{rcl}
\underset{m\rightarrow\infty}{\limsup}\hspace{0.2cm}\frac{\log \chi^{(i)}_{n+m}}{|\Delta_{n+m}|}
& \geq d^{-n-1} \left[(d-1)\log \chi^{(i)}_{n}-d\log(a_{q}/a_{i})\right].
\end{array}
\end{equation}
Then, for $1\leq i\leq q-1$ and $n\geq 1$, we have
\begin{equation*}
\begin{array}{rcl}
\underset{m\rightarrow\infty}{\limsup}\frac{\log p_{b_{i}}(m)}{|\Delta_{m}|} & = & \underset{m\rightarrow\infty}{\limsup}\frac{\log p_{b_{q}}( m)\cdot\chi^{(i)}_{m}}{|\Delta_{m}|}\\
& & \\
& = & \log a_{q} + \underset{m\rightarrow\infty}{\limsup}\frac{\log \chi^{(i)}_{m}}{|\Delta_{m}|}\\
& & \\
& = & \log a_{q} + \underset{m\rightarrow\infty}{\limsup}\frac{\log \chi^{(i)}_{n+m}}{|\Delta_{n+m}|}\\
 & & \\
 & \geq & \log a_{q} + d^{-n-1}\left[(d-1)\log \chi^{(i)}_{n}-d\log(a_{q}/a_{i})\right].
\end{array}
\end{equation*}
Therefore,  by (\ref{eqn:2.2.3}), the result follows straightforwardly. The proof is complete.
\end{proof}

In the following, we consider the reducible matrix $M$ as in (\ref{eqn:2.2.1}) with $q=3$. Let
\begin{equation*}
M=
\left[
\begin{array}{ccc}
E_{a_{1}} & R_{l_{1,2}} & R_{l_{1,3}} \\
O & E_{a_{2}} & R_{l_{2,3}} \\
O & O & E_{a_{3}}
\end{array}
\right].
\end{equation*}
In particular, when $a_{1}+l_{1,2}+l_{1,3}>a_{3}$ and $a_{2}+l_{2,3}>a_{3}$, by Lemma \ref{lemma:3.0.1}, both $\{\chi^{(1)}_{n}\}$ and $\{\chi^{(2)}_{n}\}$ are increasing, and then it will be simpler to determine whether $h(\mathcal{T}_{M})=\log a_{3}$ or not. We provide more explicit and checkable conditions in the following theorem.

\begin{theorem}
\label{theorem:3.0.2-1}
Suppose $M=M(a_{1},a_{2},a_{3};l_{1,2},l_{1,3},l_{2,3})$ and $\mathcal{T}_{M}$ is defined on $2$-tree.
Let the largest real root and second largest real root of $f(x)\equiv\left(\frac{a_{2}}{a_{3}}x+\frac{l_{2,3}}{a_{3}}\right)^{2}=x$ on $[1,\infty)$ be $x_{+}$ and $x_{-}$ respectively (if $x_{+}$ exists but $x_{-}$ does not exist, let $x_{-}=x_{+}$). For $a_{1}+l_{1,2}+l_{1,3}>a_{3}$ and $a_{2}+l_{2,3}>a_{3}$,
\begin{itemize}
\item[(a)] when $x_{+}$ does not exist or $x_{+}<1$, $h(\mathcal{T}_{M})>\log a_{3}$.
\item[(b)] when $x_{+}\geq 1$, let the maximal real root of
\begin{equation*}
g(x)\equiv\left(\frac{a_{1}}{a_{3}}x+\frac{l_{1,2}}{a_{3}}x_{-}+\frac{l_{1,3}}{a_{3}}\right)^{2}=x
\end{equation*}
be $x'_{+}$,
\begin{itemize}
\item[(i)]if $x'_{+}$ does not exist or $x'_{+}<1$, then $h(\mathcal{T}_{M})>\log a_{3}$,
\item[(ii)]if $x'_{+}\geq 1$, then  $h(\mathcal{T}_{M})=\log a_{3}$.
\end{itemize}
\end{itemize}
\end{theorem}

\begin{proof}
Let $M'=
\left[
\begin{array}{cc}
E_{a_{2}} & R_{l_{2,3}} \\
O & E_{a_{3}}
\end{array}
\right]$. By Theorem \ref{theorem:3.0.2}, we have $h(\mathcal{T}_{M})\geq h(\mathcal{T}_{M'})> \log a_{3}$ when $x_{+}$ does not exist or $x_{+}<1$.   When $x_{+}\geq 1$, by Lemma \ref{lemma:3.0.1} and  Theorem \ref{theorem:3.0.2}, $\{\chi^{(2)}_{n}\}$ approaches $x_{-}$ as $n\rightarrow\infty$ and $h(\mathcal{T}_{M'})=\log a_{3}$. Furthermore, $g(x)$ is well-defined when $x_{+}\geq 1$.

For Case (b)(i), consider
\begin{equation*}
g_{m}(x)\equiv \left(\frac{a_{1}}{a_{3}}x+\frac{l_{1,2}}{a_{3}}\chi^{(2)}_{m}+\frac{l_{1,3}}{a_{3}}\right)^{2}
\end{equation*}
for $m\geq 0$.
Let the maximal real root of $g_{m}(x)=x$ be $x'_{+}(m)$.
Since $\{\chi^{(2)}_{n}\}$ approaches to $x_{-}$ increasingly as $n\rightarrow\infty$, it can be shown that if $x'_{+}$ does not exist, then there exists $N\geq 1$ such that $x'_{+}(m)$ does not exist for $m\geq N$; if $x'_{+}<1$, then there exists $N\geq 1$ such that $x'_{+}(m)<1$  for $m\geq N$. It is clear that $\chi^{(1)}_{n+1}\geq g_{m}(\chi^{(1)}_{n})$ for $n\geq m$. By Lemma \ref{lemma:3.0.1}, it can be verified that $\{\chi^{(1)}_{n}\}$ approaches $\infty$ as $n$ tends $\infty$. Therefore, from Theorem \ref{theorem:3.0.2-0}, $h(\mathcal{T}_{M})>\log a_{3}$.

For Case (b)(ii), let $\omega_{n+1}=g(\omega_{n})$, $n\geq 0$, with $\omega_{0}=1$. By Lemma \ref{lemma:3.0.1}, $\{\omega_{n}\}$ is bounded. Since $\chi^{(1)}_{n}\leq \omega_{n}$, $n\geq 0$, $\{\chi^{(1)}_{n}\}$ is bounded, which yields $h(\mathcal{T}_{M})=\log a_{3}$. The proof is complete.
\end{proof}

Table 4 provides numerical results for each case in Theorem \ref{theorem:3.0.2-1}.
\begin{center}
\begin{tabular}{|c|c|c|c|c|c|c|c|c|}
\hline
    Case in Thm \ref{theorem:3.0.2-1} &$a_{1}$ & $a_{2}$ & $a_{3}$ & $l_{1,2}$ & $l_{1,3}$ & $l_{2,3}$ & $h(\mathcal{T}_{M})$ & $\log a_{3}$   \\
\hline
(a) &1 & 2 & 3 & 1 & 2 & 2 & 0.517166 & 0.477121  \\
\hline
(b)(i) &2 & 1 & 5 & 1 & 3 & 5 & 0.703385 & 0.69897  \\
\hline
(b)(ii) &1 & 1 & 6 & 1 & 5 & 6 & 0.778151& 0.778151  \\
\hline
\end{tabular}
\end{center}
\begin{equation*}
\text{Table 4.}
\end{equation*}

\section{Relative denseness of entropy for irreducible Markov hom tree-shift}
This section is devoted to the relative denseness of the sets
\begin{equation} \label{eq:set_of_entropy}
\mathcal{H}^{(d)}=\left\{h(\mathcal{T}_{M}):\mathcal{T}_{M} \text{ is on }d\text{-tree and }M \text{ is binary}\right\},
\end{equation}
and
\begin{equation} \label{eq:set_of_irr_entropy}
\mathcal{H}^{(d)}_{irr}=\left\{h(\mathcal{T}_{M}): \mathcal{T}_{M} \text{ is on }d\text{-tree,} \text{ and }M \text{ is binary and irreducible}\right\},
\end{equation}
in the set of entropies of all hom tree-shift. More specifically, we first demonstrate that \eqref{eq:set_of_entropy} and \eqref{eq:set_of_irr_entropy} shares a common closure and then show the closure coincides with the closure of $\mathcal{H(S)}$, where $\mathcal{S}$ denotes the the collection of all hom tree-shifts. To this end, we prove that for every hom Markov tree-shift $\mathcal{T}_M$, there exists a sequence of irreducible matrices $M_N$ such that $h(\mathcal{T}_{M_N})$ converges to $h(\mathcal{T}_M)$. Equivalently, for every hom Markov tree-shift $\mathcal{T}_G$ induced by a directed graph $G$, there exists a sequence of strongly connected graphs $G_N$ (i.e., for all different vertices $a, b$ in $G_N$, there is a path from $a$ to $b$), which is some component of $G^{(N)}$ defined in \eqref{eq:graph_union} such that $h(\mathcal{T}_{G_N})$ converges to $h(\mathcal{T}_G)$. In summary, the construction of $\mathcal{T}_{G_N}$ is summarized in the following steps:
\begin{enumerate}
    \item Summarize the graph representation of a hom Markov tree-shift by means of its component graph to extract the subtrees of components, each of which admits exactly one vertex with zero indegree.
    \item Modify the subgraphs of the original graph restricted on each of the subtrees of components to generate the desired sequence of strongly connected graphs.
\end{enumerate}
For the convenience of the reader, each step of the construction of $G^{(N)}$ is illustrated in Figure \ref{fig:construction_overview}. 
	
    Let $M \in \{0,1\}^{\norm{\alphabet} \times \norm{\alphabet}}$ be a essential adjacency Matrix indexed by $\alphabet$ with irreducible decomposition as
	\begin{equation}
		M=\begin{bmatrix}
			M_{1,1} & R_{1,2} & R_{1,3} & \cdots & R_{1,q} \\
			O & M_{2,2} & R_{2,3} & \cdots & R_{2,q} \\
			O & O & M_{3,3} & \cdots & R_{3,q} \\
			& \vdots & & \ddots & \vdots \\
			O & O & O & \cdots & M_{q,q},
		\end{bmatrix}
	\end{equation}
	where $M_{i,i}$ is a $q_i \times q_i$ irreducible submatrix indexed by $\alphabet_i=\{a_{i;j}: 1 \le j \le q_i\}$. Define 
	\begin{equation*}
	    N_i=\begin{cases}
	    1 & \text{if } M_{i,i}=O, \\
	    \min \left\{N \in \Nint: \sum\limits_{j=1}^N [(M_{i,i})^{j}]_{a,b} \ge 1, \forall a,b \in \alphabet_i \right\} & \text{otherwise.}
	    \end{cases}
	\end{equation*}
	
	\noindent \textbf{Construction of $G^{(N)}$} Let $G=(V,E)$ be the graph representation of $\mathcal{T}_M$ defined as follows.
	\begin{equation*}
		\begin{cases}
		V=\alphabet, \\
		E=\{(a,b) \in V \times V: M_{a,b}=1, \forall a,b \in \alphabet\}.
		\end{cases}
	\end{equation*}
	
	\noindent \textbf{Step 1.} Consider the graph $\mathbf{G}=(\mathbf{V},\mathbf{E})$ of $G$, which is defined as
	\begin{equation*}
		\begin{cases}
			\mathbf{V}=\{\alphabet_i: 1 \le i \le q\}, \\
			\mathbf{E}=\{(\alphabet_i,\alphabet_j) \in \mathbf{V} \times \mathbf{V}: i \ne j, R_{i,j} \ne O\}.
		\end{cases}
	\end{equation*}
    Suppose $\mathbf{I}$ (respectively, $\mathbf{T}$) are vertices in $\mathbf{G}$ whose indegrees (respectively, outdegree) are zeros. Without loss of generality, we assume $\mathbf{I} \cap \mathbf{T} = \emptyset$, for $\alphabet_i \in \mathbf{I} \cap \mathbf{T}$ corresponds to an isolated component in $G$, which can be separated from the beginning of the discussion. For each $\alphabet_{i^\ast} \in \mathbf{I}$, define the restriction $\mathbf{G}_{i^\ast}=(\mathbf{V}_{i^\ast},\mathbf{E}_{i^\ast})$ of $\mathbf{G}$ as
	\begin{align*}
		\begin{cases}
			\mathbf{V}_{i^\ast}=\{\alphabet_{i^\ast}\} \cup \{\alphabet_i: \exists \text{ a path in } \mathbf{G} \text{ from } \alphabet_{i^\ast} \text{ to } \alphabet_i\}, \\
			\mathbf{E}_{i^\ast}=\mathbf{E} \cap (\mathbf{V}_{i^\ast} \times \mathbf{V}_{i^\ast}),
		\end{cases}
	\end{align*}
	in which $\alphabet_{i^\ast}$ is the only vertex with zero indegree. Define the graphs $G_{i^\ast}=(V_{i^\ast},E_{i^\ast})$ as
	\begin{equation*}
		\begin{cases}
			V_{i^\ast}=\bigcup\limits_{\alphabet_i \in \mathbf{V}_{i^\ast}} \alphabet_i, \\
			E_{i^\ast}=E \cap (V_{i^\ast} \times V_{i^\ast}).
		\end{cases}
	\end{equation*}
	
	\noindent \textbf{Step 2.} For each $N \ge 1$, define the graph of components $G^{(N)}_{i^\ast}=(V^{(N)}_{i^\ast}, E^{(N)}_{i^\ast})$ as follows
	\begin{equation*}
		\begin{cases}
			\begin{split}
				V^{(N)}_{i^\ast} = & \ V^{(N)}_{i^\ast,1} \cup V^{(N)}_{i^\ast,2} \cup V^{(N)}_{i^\ast,3} \\
				= & \ (V_{i^\ast} \setminus \cup_{\alphabet_i \in \mathbf{T}} \alphabet_i) \times \{i^\ast\} \times \{0\} \\
				& \begin{split}
				    \cup \{(a,i^\ast,r): a \in V_{i^\ast} \cap \alphabet_i, \alphabet_i \in \mathbf{T}, 0 \le r \le 2 N, \exists \text{ a walk } &b_{-1} b_{0} \cdots b_{r-1} a \text{ in } G, \\
				    b_{-1} \in V_{i^\ast} \setminus \alphabet_i, b_{j} \in \alphabet_i, \forall \ 0 \le j \le r-1\}
				\end{split} \\
				& \begin{split} \cup \{i^\ast\} \times \{0,1, \cdots, N-1\} \end{split}
			\end{split} \\
			\begin{split}
				E^{(N)}_{i^\ast} = & E^{(N)}_{i^\ast,1} \cup E^{(N)}_{i^\ast,2} \cup E^{(N)}_{i^\ast,3} \cup E^{(N)}_{i^\ast,4} \cup E^{(N)}_{i^\ast,5} \cup E^{(N)}_{i^\ast,6} \\
				= & \begin{split} \{((a,i^\ast,0),(b,i^\ast,0)) \in V^{(N)}_{i^\ast,1} \times V^{(N)}_{i^\ast,1}: (a,b) \in E_{i^\ast} \} \end{split} \\
				& \begin{split} \cup \{((a,i^\ast,0),(b,i^\ast,0)) \in V^{(N)}_{i^\ast,1} \times V^{(N)}_{i^\ast,2}: (a,b) \in E_{i^\ast} \} \end{split}\\
				& \begin{split} \cup \{((a,i^\ast,r),(b,i^\ast,r+1)) \in V^{(N)}_{i^\ast,2} \times V^{(N)}_{i^\ast,2}: (a,b) \in E_{i^\ast}, 0 \le r < 2 N\} \end{split} \\
				& \begin{split} \cup \{((a,i^\ast,2 N),(i^\ast,0)) \in V^{(N)}_{i^\ast,2} \times V^{(N)}_{i^\ast,3}\} \end{split}\\
				& \begin{split} \cup \{((i^\ast,r),(i^\ast,r+1)) \in V^{(N)}_{i^\ast,3} \times V^{(N)}_{i^\ast,3}: 0 \le r < N-1\} \end{split}\\
				& \begin{split} \cup \{((i^\ast,N-1),(a_{i^\ast;1},i^\ast,0)) \in V^{(N)}_{i^\ast,3} \times V^{(N)}_{i^\ast,1}\} \end{split}
			\end{split}.
		\end{cases}
	\end{equation*}
	Finally, define
	\begin{equation} \label{eq:graph_union}
		G^{(N)} = (V^{(N)},E^{(N)}):= \bigcup\limits_{\alphabet_{i^\ast} \in \mathbf{I}} G^{(N)}_{i^\ast}.
	\end{equation}
	This sequence of graphs $G^{(N)}$ corresponds to an associated sequence of entropies $h(\mathcal{T}_G^{(N)})$ converging to $h(\mathcal{T}_G)$.
	
	\begin{remark}
	A few remarks on the construction of $G^{(N)}$ are noteworthy. In \textbf{Step 1}, it follows from the irreducible decomposition that $\mathbf{G}$ admits no cycles and thus is composed of oriented trees. In \textbf{Step 2}, $V^{(N)}_{i^\ast,1}$ is simply comprised of the vertices which do not lie in the component with zero outdegree. Secondly, the vertices in $V^{(N)}_{i^\ast,2}$ are essentially multiple copies of the vertices lying in the component with zero outdegree, which are arranged in a pipeline with $(2N+1)$ phases so that the indegree and the outdegree of each vertex are at least one. Finally, $V^{(N)}_{i^\ast,3}$ entailing $V^{(N)}_{i^\ast,2}$ consists of merely $N$ dummy vertices for an extra extension so as to render the graph a strongly connected graph.
	\end{remark}
	\begin{example}
	    An example is provided in Figure \ref{fig:construction_overview}, in which the matrix $M$ is given in Figure \ref{fig:adjacency_matrix} and the corresponding illustration of each graph is shown in Figure \ref{fig:graph_representation}, \ref{fig:component_tree}, \ref{fig:separated_component_tree}, \ref{fig:graph_representation_separated_tree}, and \ref{fig:graph_extension}. The components in $\mathbf{I}$ are highlighted in green, while the ones in $\mathbf{T}$ are highlighted in blue.
	\end{example}
	
	\begin{figure}
	    \centering
	    \begin{subfigure}[t]{0.45\textwidth}
    	    \centering
    	    \includegraphics[scale=0.7]{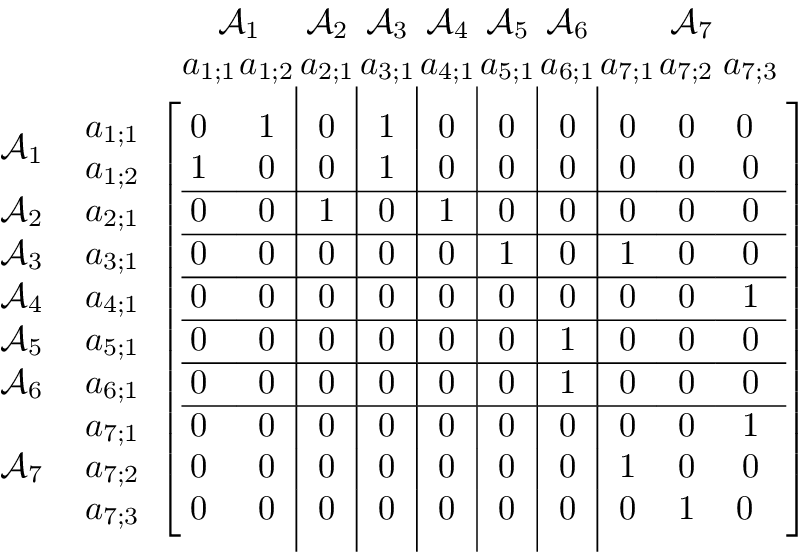}
    	    \caption{Adjacency matrix $M$}
    	    \label{fig:adjacency_matrix}
    	\end{subfigure}
	    \begin{subfigure}[t]{0.45\textwidth}
    	    \includegraphics[width=0.9\textwidth]{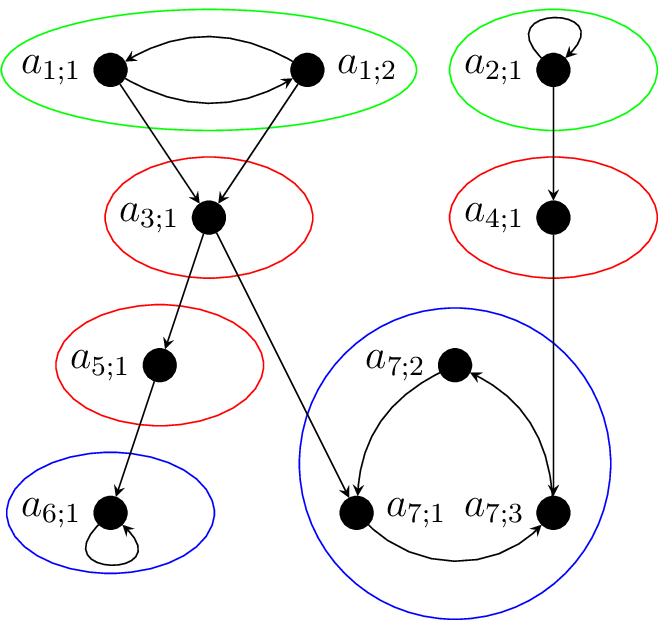}
    	    \caption{Illustration of $G=(V,E)$}
    	    \label{fig:graph_representation}
	    \end{subfigure}
	    \begin{subfigure}[t]{0.45\textwidth}
    	    \centering
    	    \includegraphics{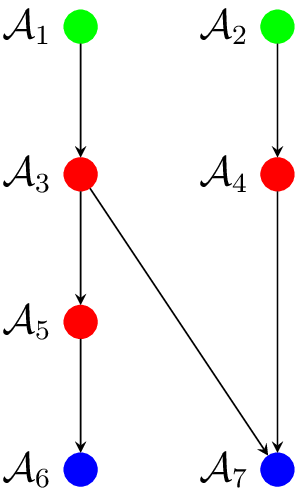}
    	    \caption{Illustration of $\mathbf{G}=(\mathbf{V},\mathbf{E})$}
    	    \label{fig:component_tree}
    	\end{subfigure}
    	\begin{subfigure}[t]{0.45\textwidth}
    	    \centering
    	    \includegraphics{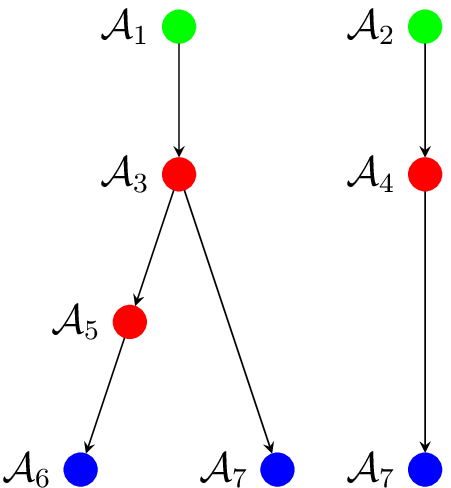}
    	    \caption{Illustration of $\mathbf{G}_{i^\ast}=(\mathbf{V}_{i^\ast},\mathbf{E}_{i^\ast})$}
    	    \label{fig:separated_component_tree}
    	\end{subfigure}
    	\begin{subfigure}[t]{0.5\textwidth}
    	    \centering
    	    \includegraphics[width=0.9\textwidth]{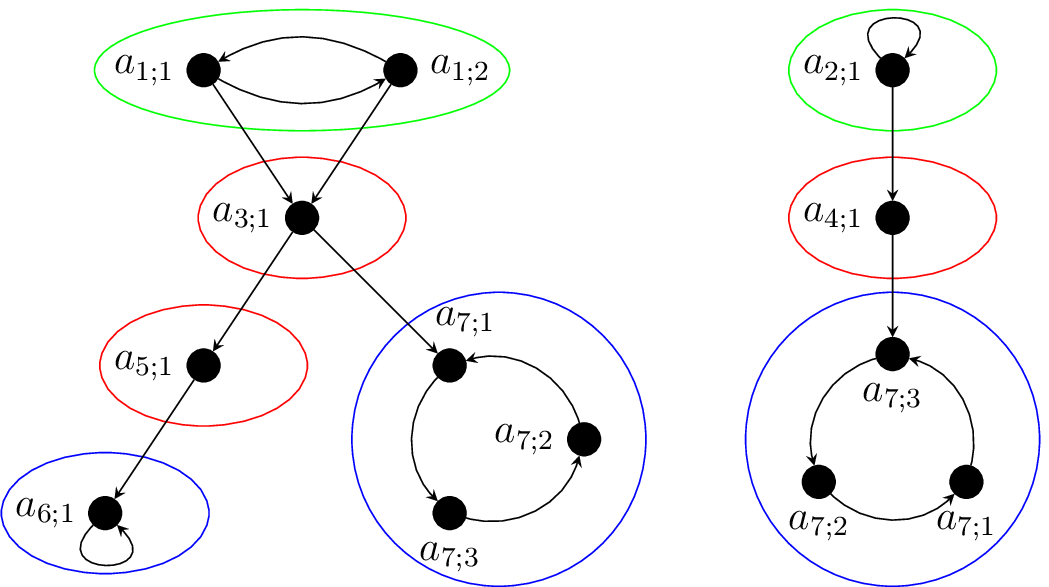}
    	    \caption{Illustration of $G_{i^\ast}(V_{i^\ast},E_{i^\ast})$}
    	    \label{fig:graph_representation_separated_tree}
    	\end{subfigure}
    	\begin{subfigure}[t]{0.45\textwidth}
    	    \centering
    	    \includegraphics[scale=0.45]{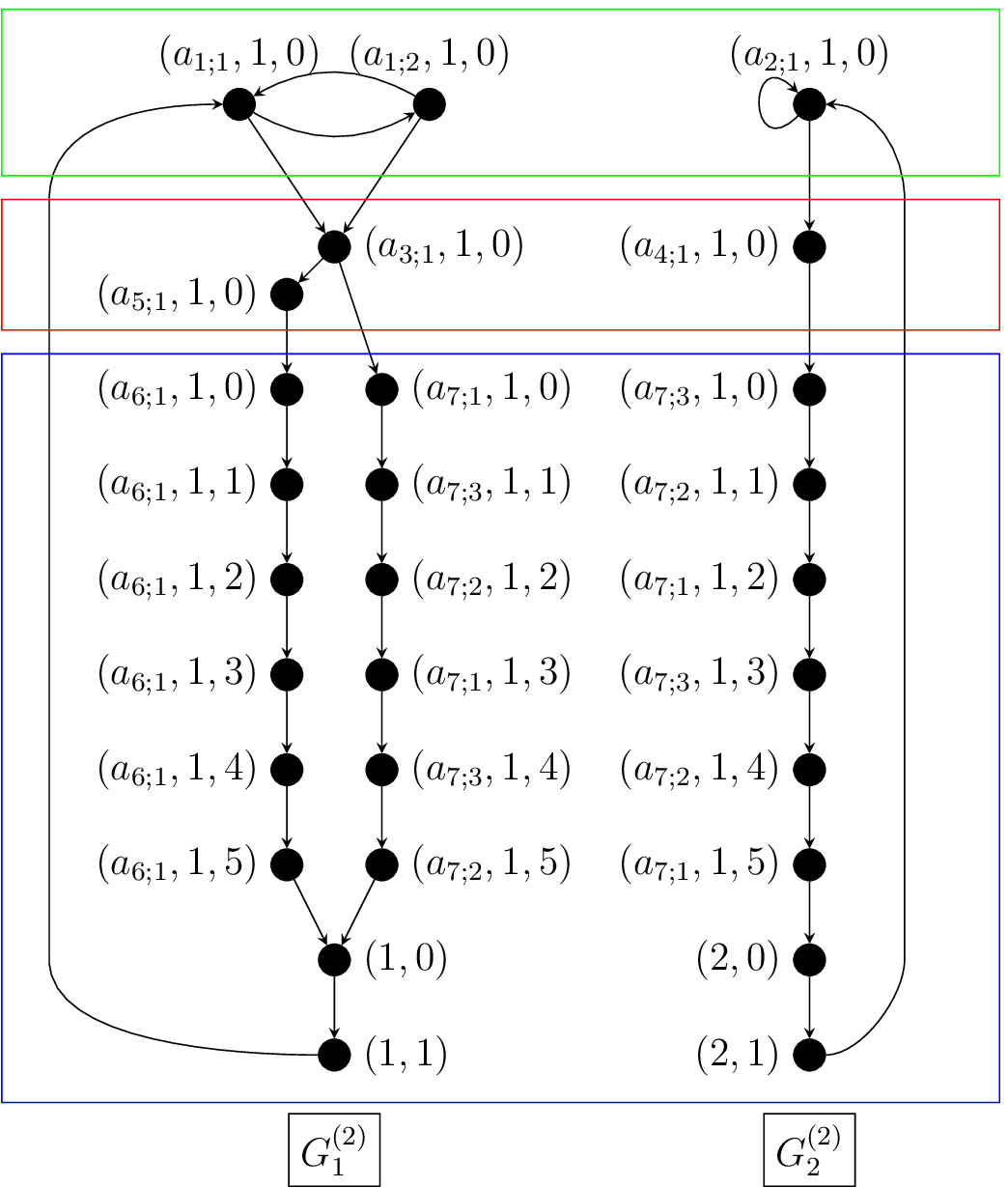}
    	    \caption{Illustration of $G^{(N)}=(V^{(N)},E^{(N)})$}
    	    \label{fig:graph_extension}
    	\end{subfigure}
    	\caption{Illustration of construction of $G^{(N)}$}
    	\label{fig:construction_overview}
	\end{figure}
	\begin{proposition} \label{prop:block_comparison}
		Suppose $N \ge \max_i N_i$ and $n \le N$. Then,
		\item[(1)] For each $a \in \alphabet$, there exist $i^\ast \in \mathbf{I}$ and $0 \le r_0 \le \max_i N_i$ such that $p_{G;a}(n) \le p_{G^{(N)};(a,i^\ast,r_0)}(n)$.
		\item[(2)] $p_{G^{(N)}}(N) \le p_G(N) \cdot (3 N + 2) \cdot \norm{\mathbf{I}}$.
	\end{proposition}
	\begin{proof}
		Suppose $f:V^{(N)} \to (\alphabet \cup \mathbf{I})$ is defined such that $f(a,i^\ast,r)=a$ and that $f(i^\ast,r)=i^\ast$. We could further define an induced map $f^\ast:\mathcal{B}_n(\mathcal{T}_{G^{(N)}}) \to (\alphabet \cup \mathbf{I})^{\Delta_n}$ for all $n \ge 0$ such that $(f^\ast(v))_g=f(v_g)$.
		
		(1) To prove this, it is sufficient to show that for every $u \in \mathcal{B}_n(\mathcal{T}_G)$, there exists $v(u) \in \mathcal{B}_n(\mathcal{T}_{G^{(N)}})$ such that $f^\ast(v(u))=u$ and that $(v(u))_\epsilon = (u_\epsilon,i^\ast,r_0)$ for some $r_0$ depending only on $u_\epsilon$. We prove the claim by induction. For the the case $n=0$, the induction hypothesis holds naturally, since if $u=a$, then there exist, by definition of $N_i$ and the formulation above, $\alphabet_{i^\ast} \in \mathbf{I}$ and $0 \le r_0 \le \max_i N_i$ such that $v(u)=(a,i^\ast,r_0)$. We suppose the induction hypothesis holds for $n$ and $u \in \mathcal{B}_{n+1}(\mathcal{T}_G)$ is given. For the case where $n+1 \le N$, we construct $v(u) \in \mathcal{B}_{n+1}(\mathcal{T}_{n+1})$ satisfying $f^\ast(v(u))=u$ and $(v(u))_\epsilon = (u_\epsilon,i^\ast,r_0)$ with $0 \le r_0 \le \max_i N_i$. It follows from the induction hypothesis that there exists $v(u|_{\Delta_n}) \in \mathcal{B}_n(\mathcal{T}_{G^{(N)}})$ satisfying $f^\ast(v(u|_{\Delta_n})) = u|_{\Delta_n}$ and $(v(u|_{\Delta_n}))_\epsilon = (u_\epsilon,i^\ast,r_0)$ with $0 \le r_0 \le \max_i N_i$. Since $n+1 \le N$, $(v(u|_{\Delta_n}))_g$ has the form $(a,i^\ast,r)$ with $r_0 \le r \le 2 \cdot \max_i N_i - 1$ and we may define $v(u)$ as follows:
		\begin{equation*}
			(v(u))_g:=
			\begin{cases}
				(v(u|_{\Delta_n}))_{g} & \text{if } \norm{g} \le n, \\
				(u_{g},i^\ast,0) & \text{if } g=g' w, \norm{g'}=n, v(u|_{\Delta_n})_{g'}=(a,i^\ast,0), a \in \alphabet_i, \alphabet_i \notin \mathbf{T}, \\
				(u_{g},i^\ast,r+1) & \text{if } g=g' w, \norm{g'}=n, v(u|_{\Delta_n})_{g'}=(a,i^\ast,r), a \in \alphabet_i, \alphabet_i \in \mathbf{T}. \\
			\end{cases}
		\end{equation*}
		It is clear that $v^{(N)} \in \mathcal{B}_{n+1}(\mathcal{T}_{G^{(N)}})$, that $f^\ast(v(u))=u$, and that $r_0$ depends solely on $u_\epsilon$. Hence, the claim holds for all $0 \le n \le N$.
		
		(2) We subdivide $\mathcal{B}_N(\mathcal{T}_{G^{(N)}})$ into the following four disjoint subsets, and estimate the cardinality of each one.
		\begin{align*}
			\mathcal{B}_N(\mathcal{T}_{G^{(N)}}) = & S_1 \cup S_2 \cup S_3 \cup S_4 \\
			= & \{v \in \mathcal{B}_N(\mathcal{T}_{G^{(N)}}): v_\epsilon = (a,i^\ast,0), a \in \alphabet_i, \alphabet_i \notin \mathbf{T}\} \\
			& \cup \{v \in \mathcal{B}_N(\mathcal{T}_{G^{(N)}}): v_\epsilon = (a,i^\ast,r), a \in \alphabet_i, \alphabet_i \in \mathbf{T}, 0 \le r \le N\} \\
			& \cup \{v \in \mathcal{B}_N(\mathcal{T}_{G^{(N)}}): v_\epsilon = (a,i^\ast,r), a \in \alphabet_i, \alphabet_i \in \mathbf{T}, N+1 \le r \le 2 N\} \\
			& \cup \{v \in \mathcal{B}_N(\mathcal{T}_{G^{(N)}}): v_\epsilon = (i^\ast,r), 0 \le r \le N-1\}.
		\end{align*}
		
		Note that if $v \in S_1$, $f^\ast(v) \in \mathcal{B}_N(\mathcal{T}_G)$. Furthermore, given $u \in \mathcal{B}_N(\mathcal{T}_G)$, 
		\begin{equation*}
			\norm{\{v \in S_1: f^\ast(v)=u\}} \le \norm{\mathbf{I}}.
		\end{equation*}
		We have
		\begin{equation} \label{eq:pattern_estimation_1}
			\norm{S_1} \le \norm{\mathbf{I}} \cdot \norm{p_{G}(N)}.
		\end{equation}
		
		The estimate of $S_2$ is similar to $S_1$. If $v \in S_2$, $f^\ast(v) \in \mathcal{B}_N(\mathcal{T}_G)$. On the other hand, from the definition of $V^{(N)}_{i^\ast,2}$ it follows that
		\begin{equation*}
			\norm{\{v \in S_2: f^\ast(v)=u\}} \le \norm{\mathbf{I}} \cdot (N+1).
		\end{equation*}
		Thus, we obtain
		\begin{equation} \label{eq:pattern_estimation_2}
			\norm{S_2} \le \norm{\mathbf{I}} \cdot (N+1) \cdot \norm{p_{G}(N)}.
		\end{equation}
		
		As for $S_3$, every $v \in S_3$ has the form in Figure \ref{fig:pattern_1}. In particular, $f^\ast(v|_{\Delta_{2 N - r}}) \in \mathcal{B}_{2 N - r}(\mathcal{T}_{G})$. Also, $v_g = v'_g$ for all $\norm{g} > 2N-r$ if $v, v' \in \mathcal{B}_N(\mathcal{T}_{G_N})$, $v_\epsilon = (a,i^\ast,r)$, $v'_\epsilon = (b,i^\ast,r)$. Hence,
		\begin{equation} \label{eq:pattern_estimation_3}
    		\begin{aligned}
    			\norm{S_3} & = \sum_{\alphabet_{i^\ast} \in \mathbf{I}} \sum_{r=N+1}^{2 N} \sum_{a \in \alphabet_{i^\ast}} p_{G^{(N)};(a,i^\ast,r)}(2 N - r) \\
    			& = \sum_{\alphabet_{i^\ast} \in \mathbf{I}} \sum_{r=N+1}^{2 N} \sum_{a \in \alphabet_{i^\ast}} p_{G;a}(2 N - r) \\
    			& = \sum_{\alphabet_{i^\ast} \in \mathbf{I}} \sum_{r=N+1}^{2 N} \sum_{a \in \alphabet_{i^\ast}} p_{G}(2 N - r) \\
    			& \le \norm{\mathbf{I}} \cdot N \cdot p_{G}(N).
			\end{aligned}
		\end{equation}
		
		Finally, every $v \in S_4$ has the form in Figure \ref{fig:pattern_2}. In particular, $f^\ast((\sigma_g v)|_{\Delta_{r}}) \in \mathcal{B}_{r}(\mathcal{T}_{G})$ with $v_g=a_{i^\ast;1}$ if $\norm{g} = N-r$. Also, $v|_{\Delta_{N-r}} = v'|_{\Delta_{N-r}}$ if $v_\epsilon = v'_\epsilon = (i^\ast,r)$. On the other hand, since $G$ is essential, for every $0 \le r \le N-1$ there is an walk $a_{i_{r};j_{r}} a_{i_{r+1};j_{r+1}} \cdots a_{i_{N-1};j_{N-1}} a_{i^\ast;1}$ in $G$. Hence, we are able to define $u(v) \in \mathcal{B}_{N}(\mathcal{T}_G)$, as illustrated in Figure \ref{fig:pattern_3}, such that
		\begin{equation*}
			(u(v))_g  = \begin{cases}
				a_{i_{r+\norm{g}};j_{r+\norm{g}}} & \text{if } 0 \le \norm{g} \le N-r-1,\\
				f(v_g) & \text{otherwise.}
			\end{cases}
		\end{equation*}
		Note that it follows immediately from the definition that if $v_\epsilon = v'_\epsilon$ yet $v \ne v'$, then $u(v) \ne u(v')$. Hence, 
		\begin{equation} \label{eq:pattern_estimation_4}
		    \begin{aligned}
    			\norm{S_4} & = \sum_{\alphabet_{i^\ast} \in \mathbf{I}} \sum_{r=0}^{N-1} p_{G^{(N)};(i^\ast,r)}(N) \\
    			& = \sum_{\alphabet_{i^\ast} \in \mathbf{I}} \sum_{r=0}^{N-1} (p_{G^{(N)};(a_{i^\ast,1};i^\ast,0)}(r))^{d^{N-r}} \\
    			& \le \sum_{\alphabet_{i^\ast} \in \mathbf{I}} \sum_{r=0}^{N-1} p_{G;a_{i_r;j_r}}(N) \\
    			& \le \norm{\mathbf{I}} \cdot N \cdot p_{G}(N). \\
    		\end{aligned}
		\end{equation}
		Combining \eqref{eq:pattern_estimation_1}, \eqref{eq:pattern_estimation_2}, \eqref{eq:pattern_estimation_3} with \eqref{eq:pattern_estimation_4} yields the result in (2). The proof is then completed.
		\begin{figure}
		    \centering
		    \includegraphics[scale=0.7]{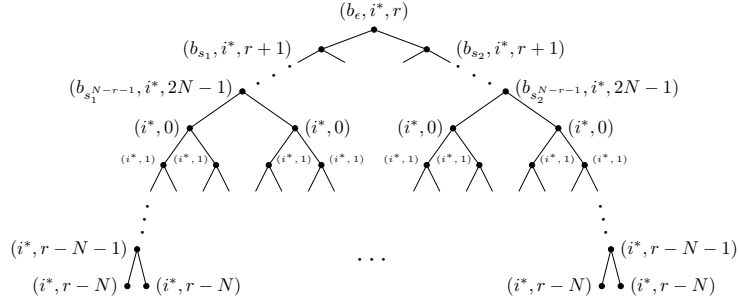}
		    \caption{$N$-block $v \in S_3$}
		    \label{fig:pattern_1}
		\end{figure}
		\begin{figure}
		    \centering
		    \includegraphics[scale=0.7]{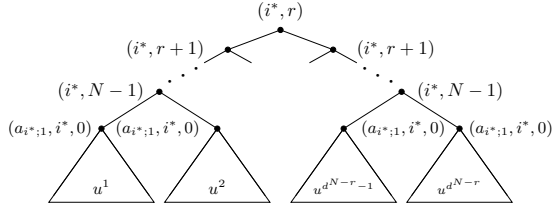}
		    \caption{$N$-block $v \in S_4$}
		    \label{fig:pattern_2}
		\end{figure}
		\begin{figure}
		    \centering
		    \includegraphics[scale=0.7]{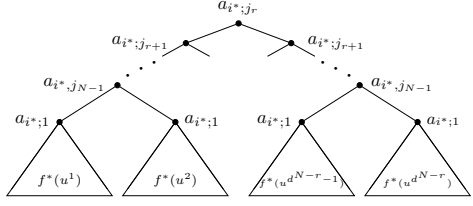}
		    \caption{The corresponding $N$-block $u^{(v)}$ of $v \in S_4$}
		    \label{fig:pattern_3}
		\end{figure}
	\end{proof}
	\begin{proposition} \label{prop:entropy_convergence_essentiality}
		For every essential graph $G$, we have
		\begin{equation} \label{eq:entropy_series_representation}
			\frac{\log \max_a p_{G;a}(n)}{d^{n+1}/(d-1)} = \sum_{i=0}^{n} \frac{d-1}{d^{i+1}} \log \gamma_n,
		\end{equation}
		for some $1 \le \gamma_i \le \norm{\alphabet}^d$. Moreover, $\frac{\log \max_a p_{G;a}(n)}{d^{n+1}/(d-1)}$ tends to $h(\mathcal{T}_G)$ increasingly. 
	\end{proposition}
	\begin{proof}
		Suppose $M \in \{0,1\}^{\norm{\alphabet} \times \norm{\alphabet}}$ be the adjacency matrix of $G$ and write $\mathbf{p}_G(n) = [p_{G;a_1}(n), p_{G;a_2}(n), \cdots, p_{G;a_{\norm{{A}}}}(n)]^T \in \mathbb{R}^{\norm{\alphabet}}$. It is known (see \cite{BC-apa2015}) that
		\begin{equation*}
			\begin{cases}
				\mathbf{p}_{G}(n) = \bigodot\limits_{i=1}^d (M \mathbf{p}_{G}(n-1)), \\
				\mathbf{p}_{G}(0) = [1,1,\cdots,1]^T,
			\end{cases}
		\end{equation*}
		where $\odot$ denote the entrywise product of the column vectors. We claim that for any given positive real sequence $\{\gamma_i\}_{i=0}^{\infty}$, the system
		\begin{equation*}
			\begin{cases}
				\overline{\mathbf{p}}_{G}(n) = \frac{1}{\gamma_n} \bigodot\limits_{i=1}^d (M \overline{\mathbf{p}}_{G}(n-1)), \\
				\overline{\mathbf{p}}_{G}(0) = [\frac{1}{\gamma_0},\frac{1}{\gamma_0},\cdots,\frac{1}{\gamma_0}]^T,
			\end{cases}
		\end{equation*}
		has the following property:
		\begin{equation*}
			\mathbf{p}_{G}(n) = \gamma_{0}^{d^n} \gamma_{1}^{d^{n-1}} \cdots \gamma_{n}^{d^{0}} \overline{\mathbf{p}}_{G}(n) , \forall n \ge 0.
		\end{equation*}
		We prove the claim by induction. The case $n=0$ is immediate. We now verify the case $n=N+1$ by assuming the induction hypothesis holds for $n=N$. Indeed,
		\begin{align*}
			\mathbf{p}_G(N+1) & = \textstyle\bigodot\limits_{j=1}^d (M \mathbf{p}_G(N)) \\
			& = \textstyle\bigodot\limits_{j=1}^d (M \gamma_{0}^{d^N} \gamma_{1}^{d^{N-1}} \cdots \gamma_{N}^{d^{0}} \overline{\mathbf{p}}_G(N)) \\
			& = \textstyle\gamma_{0}^{d^{N+1}} \gamma_{1}^{d^{N}} \cdots \gamma_{N}^{d^{1}} \bigodot\limits_{j=1}^d (M \overline{\mathbf{p}}_G(N)) \\
			& = \textstyle\gamma_{0}^{d^{N+1}} \gamma_{1}^{d^{N}} \cdots \gamma_{N}^{d^{1}} \gamma_{N+1}^{d^{0}} \overline{\mathbf{p}}_G(N+1).
		\end{align*}
		Thus, the claim holds by induction.
		
		Next, we show by induction on $n$ that if $\gamma_0=1$ and $\gamma_n=\max \odot_{i=1}^d (M \overline{\mathbf{p}}_G(n-1))$ is chosen iteratively, then 
		\begin{equation} \label{eq:coeff_range}
		    1 \le \max_{1 \le i \le \norm{\alphabet}} [\odot_{i=1}^d (M \overline{\mathbf{p}}_G(n))]_i \le \norm{\alphabet}^d,
		\end{equation}
		for all $n \ge 0$. As for the case $n=0$, \eqref{eq:coeff_range} is attained from direct calculation. Now suppose \eqref{eq:coeff_range} holds for $n=N$. By the definition of $\gamma_{N+1}$, one deduces that $\max_{1 \le i \le \norm{\alphabet}} [\overline{\mathbf{p}}_G(N+1)]_i = 1$. Since every vertex in $G$ has both indegree and outdegree at least 1, $M$ is also essential. Thus, $\max_{1 \le i \le \norm{\alphabet}} [M \overline{\mathbf{p}}_G(N+1)]_i \ge 1$ and $\max_{1 \le i \le \norm{\alphabet}} [\odot_{j=1}^d (M \overline{\mathbf{p}}_G(N+1))]_i \ge 1$. On the other hand, we derive the following comparison
		\begin{equation*}
			\max_{1 \le i \le \norm{\alphabet}} \left[{\textstyle\bigodot\limits_{j=1}^d} (M \overline{\mathbf{p}}_G(N+1))\right]_i \le \max_{1 \le i \le \norm{\alphabet}} \left[{\textstyle\bigodot\limits_{j=1}^d} (M [1,1,\cdots,1]^T)\right]_i \le \norm{\alphabet}^d.
		\end{equation*}
		Hence, the claim holds by induction.
		
		Finally, an immediate consequence of the claim is that if $\gamma_n$ is chosen as above, then
		\begin{equation*}
			\max_a p_{G;a}(n) = \max_{1 \le i \le \norm{\alphabet}} \gamma_{0}^{d^n} \gamma_{1}^{d^{n-1}} \cdots \gamma_{n}^{d^{0}} \left[\overline{\mathbf{p}}_{G}(n)\right]_{i} = \gamma_{0}^{d^n} \gamma_{1}^{d^{n-1}} \cdots \gamma_{n}^{d^{0}}.
		\end{equation*}
		As a result, 
		\begin{equation*}
			\frac{\log \max_a p_{G;a}(n)}{d^{n+1}/(d-1)} = \sum_{i=0}^{n} \frac{d-1}{d^{i+1}} \log \gamma_i.
		\end{equation*}
		The proof is finished by noting that the above series is increasing and that
		\begin{equation*}
		    \lim_{n \to \infty} \frac{\log \max_a p_{G;a}(n)}{d^{n+1}/(d-1)} = \lim_{n \to \infty} \frac{\log p_{G}(n)}{\norm{\Delta_n}} = h(\mathcal{T}_{G}).
		\end{equation*}
	\end{proof}
	\begin{proposition} \label{prop:entropy_limit_comparison}
		Let $G^{(N)}$ be as defined. Then, $h(\mathcal{T}_{G_N})$ converges to $h(\mathcal{T}_G)$ as $N$ tends to infinity.
	\end{proposition}
	\begin{proof}
		Let $\varepsilon>0$ be given. Then, there exists $\overline{N} \ge \max_i N_i$ such that if $N \ge \overline{N}$, 
		\begin{equation} \label{eq:approximation_assumption_1}
		    \frac{\log \left[(3 N + 2) \cdot \norm{\mathbf{I}}\right]}{\norm{\Delta_{N}}} \le \frac{1}{2} \varepsilon, 
		\end{equation}
		\begin{equation} \label{eq:approximation_assumption_2}
		    \frac{\log p_G(N)}{\norm{\Delta_N}} \le h(\mathcal{T}_G)+\frac{1}{2} \varepsilon,
		\end{equation}
		and
		\begin{equation} \label{eq:approximation_assumption_3}
		    \frac{\log \max_a p_{G;a}(N)}{d^{N+1}/(d-1)} \ge h(\mathcal{T}_G) - \varepsilon. 
		\end{equation}
		For $N \ge \overline{N}$, combining  \ref{prop:block_comparison} with \eqref{eq:approximation_assumption_1},\eqref{eq:approximation_assumption_2} and \eqref{eq:approximation_assumption_3}, we conclude that
		\begin{align*}
			h(\mathcal{T}_{G}) & \ge \frac{\log p_G(N)}{\norm{\Delta_{N}}} - \frac{1}{2} \varepsilon \\
			& \ge \frac{\log p_{G^{(N)}}(N)}{\norm{\Delta_{N}}} - \frac{\log \left[(3 N + 2) \cdot \norm{\mathbf{I}}\right]}{\norm{\Delta_{N}}} - \frac{1}{2} \varepsilon \\
			& \ge h(\mathcal{T}_{G^{(N)}}) - \varepsilon, 
		\end{align*}
		and from Proposition \ref{prop:entropy_convergence_essentiality} that
		\begin{align*}
			h(\mathcal{T}_{G^{(N)}}) & \ge \frac{\log \max_b p_{G^{(N)};b}(N)}{d^{N+1}/(d-1)} \\
			& \ge \frac{\log \max_a p_{G;a}(N)}{d^{N+1}/(d-1)} \\
			& \ge h(\mathcal{T}_{G}) - \varepsilon.
		\end{align*}
		Since $\varepsilon$ is arbitrary, the proof is completed.
	\end{proof}
	
	To conclude the section, it is left to demonstrate the relative denseness of entropies of irreducible hom Markov tree-shifts within those of all hom Markov tree-shifts.
	\begin{theorem} \label{thm:denseness_irreducible_shift}
		Let $\overline{\mathcal{H}^{(d)}_{irr}}$ and $\overline{\mathcal{H}^{(d)}}$ is as defined in \eqref{eq:set_of_entropy} and \eqref{eq:set_of_irr_entropy}. Then, 
		\[
		\overline{\mathcal{H}^{(d)}_{irr}} = \overline{\mathcal{H}^{(d)}},
		\]
		where $\overline{A}$ denotes the closure of a set $A$.
	\end{theorem}
	\begin{proof}
	    The inclusion $\overline{\mathcal{H}^{(d)}} \supset \overline{\mathcal{H}^{(d)}_{irr}}$ is clear. To obtain the other inclusion, one finds for every $M$ a sequence of binary irreducible matrices $M_N$ satisfying $\lim_{N \to \infty} h(\mathcal{T}_{M_N}) = h(\mathcal{T}_M)$. To this end, we suppose $G$ and $G^{(N)}$ are defined for $M$ as before. Applying Proposition \ref{prop:entropy_limit_comparison} we obtain that $\lim_{N \to \infty} h(\mathcal{T}_{G^{(N)}}) = h(\mathcal{T}_G) = h(\mathcal{T}_M)$. Since $G^{(N)} = \cup_{\alphabet_{i^\ast} \in \mathbf{I}} G^{(N)}_{i^\ast}$, in which $G^{(N)}_{i^\ast}$ are isolated strongly connected components, there must exists $i_N$ such that $h(\mathcal{T}_{G^{(N)}_{i_N}}) = h(\mathcal{T}_{G^{(N)}})$. The proof is finished by choosing matrix representation of $\mathcal{T}_{G^{(N)}_{i^\ast_N}}$ as $M_N$.
	\end{proof}
	
	\begin{corollary} \label{cor:universal_closure}
	    The following two closures of sets are equal.
	    \begin{equation*}
	        \overline{\mathcal{H}^{(d)}_{irr}} = \overline{\{h(\mathcal{T}_X):X \text{ is a shift space}\}}.
	    \end{equation*}
	\end{corollary}
	\begin{proof}
	    In \cite[Corollary 1]{petersen2020entropy}, Petersen and Salama show that for every shift space $X$, there exists a sequence of Markov shifts $\{X_r: r \in \Nint\}$ which are $1$-step higher block representation of $r$-step Markov shifts that shares a common set of $(r+1)$-blocks as $X$, such that
	    \begin{equation*}
	        \inf\limits_{r} h(\mathcal{T}_{X_r}) = h(\mathcal{T}_X).
	    \end{equation*}
	    Thus,
	    \begin{equation} \label{eq:entropy_inclusion}
	    \{h(\mathcal{T}_X): X \text{ is a shift space}\} \subset \overline{\mathcal{H}^{(d)}}.
	    \end{equation}
	    Combining \eqref{eq:entropy_inclusion} with Theorem \ref{thm:denseness_irreducible_shift}, we deduce the corollary.
	\end{proof}
	
	\begin{remark}
	    Note that $h(\mathcal{T}_G) = \max_{\alphabet_{i^\ast} \in \mathbf{I}} h(\mathcal{T}_{G_{i^\ast}})$. The argument in this section can be simplified in the way that all the discussions are restricted to a single subgraph $G_{i^\ast}$ whose induced space has the same entropy as $\mathcal{T}_{G_{i^\ast}}$.
	\end{remark}
	
\section{Denseness of $h(\mathcal{T}_{M})$}

This section investigates the denseness of $\mathcal{H}^{(d)}_{irr}$, or equivalently the denseness of $\mathcal{H}^{(d)}$. We first introduce some necessary notations. For $l\geq1$, let $\mathbf{y}_{l}=\{y_{i}\}_{i=1}^{l}$ be a positive integer sequence. Given $\mathbf{y}_{l}$ and $N\geq 1$, we defined a directed graph $G=G_{\mathbf{y}_{l};N}$ as follows. The vertex set $V(G_{\mathbf{y}_{l};N})$ of $G_{\mathbf{y}_{l};N}$ is defined by
\begin{equation*}
V(G_{\mathbf{y}_{l};N})=\bigcup_{i=1}^{N+l}V_{i}=\bigcup_{i=1}^{N+l}\left\{v_{i;m(i)}: 1\leq m(i)\leq n_{i}\right\},
\end{equation*}
where
\begin{equation}\label{eqn:2.2-1}
n_{i}=\left\{
\begin{array}{ll}
1 & \text{for }1\leq i\leq N,\\
y_{N+l+1-i} & \text{for }N+1\leq i\leq N+l;
\end{array}
\right.
\end{equation}
and the edge set  $E(G_{\mathbf{y}_{l};N})$ of $G_{\mathbf{y}_{l};N}$ is defined by
\begin{equation*}
E(G_{\mathbf{y}_{l};N})=\overset{N+l}{\underset{i=1}{\bigcup}}\left\{(v_{\alpha},v_{\beta}):v_{\alpha}\in V_{i} \text{ and }  :v_{\beta}\in V_{i+1}\right\},
\end{equation*}
where
$V_{N+l+1}=V_{1}$, and $(v_{\alpha},v_{\beta})$. Clearly, $G_{\mathbf{y}_{l};N}$ is strongly connected.

For example, let $\mathbf{y}_{4}=\{4,2,1,2\}$ and $N=2$. We have $n_{1}=1$, $n_{2}=1$, $n_{3}=y_{4}=2$, $n_{4}=y_{3}=1$, $n_{5}=y_{2}=2$ and $n_{6}=y_{1}=4$. Then,
\begin{equation*}
V(G_{\mathbf{y}_{4};2})=\left\{v_{1;1},v_{2;1},v_{3;1},v_{3;2},v_{4;1},v_{5;1},v_{5;2},v_{6;1},v_{6;1},v_{6;3},v_{6;4}\right\},
\end{equation*}
and the graph $G_{\mathbf{y}_{4};2}$ can be drawn as in Figure \ref{fig:everywhere_dense_graph}.

\begin{figure}
\psfrag{a}{{\footnotesize $v_{1;1}$}}
\psfrag{b}{{\footnotesize $v_{2;1}$}}
\psfrag{c}{{\footnotesize $v_{3;1}$}}
\psfrag{d}{{\footnotesize $v_{3;2}$}}
\psfrag{e}{{\footnotesize $v_{4;1}$}}
\psfrag{f}{{\footnotesize $v_{5;1}$}}
\psfrag{g}{{\footnotesize $v_{5;2}$}}
\psfrag{h}{{\footnotesize $v_{6;1}$}}
\psfrag{k}{{\footnotesize $v_{6;2}$}}
\psfrag{l}{{\footnotesize $v_{6;3}$}}
\psfrag{m}{{\footnotesize $v_{6;4}$}}
\includegraphics[scale=1.4]{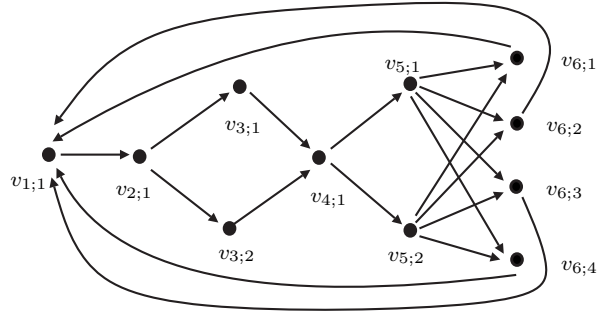}
\caption{Graph representation for entropy approximation}
\label{fig:everywhere_dense_graph}
\end{figure}

Given $n\geq 1$, for any finite sequence $\mathbf{b}_{n}=(b_{1},b_{2},\cdots,b_{n})$, we define the (left) shift map $\sigma(\mathbf{b}_{n})\equiv (b_{2},b_{3},\cdots,b_{n},b_{1})$.
Clearly, $\sigma^{n}(\mathbf{b}_{n})=\mathbf{b}_{n}$. The set of all $\sigma^{m}(\mathbf{b}_{n})$, $1\leq m\leq n$, of $\mathbf{b}_{n}$ is denoted by $\mathcal{S}(\mathbf{b}_{n})$. The main result is shown in Theorem \ref{theorem:2.2.1}.
\begin{theorem}
\label{theorem:2.2.1}
For $d\geq 2$ and any finite positive integer sequence $\mathbf{y}_{l}=(y_{1},y_{2},\cdots,y_{l})$,
\begin{equation}\label{eqn:2.3}
\mu^{(d)}_{\mathbf{y}_{l}}\equiv(d-1)\underset{\mathbf{w}_{l}\in \mathcal{S}(\mathbf{y}_{l})}{\max}\left\{\frac{1}{d}\log w_{1}+\frac{1}{d^{2}}\log w_{2}+\cdots +\frac{1}{d^{l}}\log w_{l}\right\}
\end{equation}
is an accumulation point of $\mathcal{H}^{(d)}_{irr}$. Furthermore, $\mathcal{H}^{(d)}_{irr}$ is dense in
\begin{equation}\label{eqn:2.3}
 \left[(d-1)\log2,\infty\right).
\end{equation}

\end{theorem}

\begin{proof}
Let $\mathbf{y}_{l}=\{y_{i}\}_{i=1}^{l}$ be a positive integer sequence. For $N\geq 1$, the irreducible graph $G=G_{\mathbf{y}_{l};N}$ can be constructed as above. We label the vertices with $\{1,2,\cdots, K\}$, $K=\left|V(G)\right|$. For convenience, $v_{1;1}$ is labeled by $1$. The graph representation $G$ is strongly connected. For $1\leq i\leq N+l$, $n_{i}$ is defined as (\ref{eqn:2.2-1}); for $i\geq N+l+1$, $n_{i}=n_{r}$ where $i\equiv r (\text{mod } N+l)$ and $n_{0}=n_{N+l}$. Then, by the construction of $G$, it is seen that
\begin{equation*}
p_{G;1}(n)=\overset{n}{\underset{i=0}{\prod}} n_{i+1}^{d^{i}}.
\end{equation*}
From \cite{petersen2018tree}, irreducibility of $G$ implies
\begin{equation*}
h\left(\mathcal{T}_{G}\right)= \underset{n\rightarrow\infty}{\limsup}\frac{\log p_{G;1}(n) }{|\Delta_{n}|}.
\end{equation*}
Let the finite sequence $\mathbf{n}_{N+l}=(n_{N+l},n_{N+l-1},\cdots, n_{1})$.  Since $n_{i}=1$ for $1\leq i\leq N$ and $n_i=y_{N+l+1-1}$ for $N+1\leq i\leq N+l$,
\begin{equation*}
\begin{array}{rcl}
h\left(\mathcal{T}_{G}\right) & =  & \underset{n\rightarrow\infty}{\limsup}\frac{\log p_{G;1}(n) }{|\Delta_{n}|} \\
& & \\
& = & \underset{n\rightarrow\infty}{\limsup}\frac{1}{|\Delta_{n}|} \overset{n}{\underset{i=0}{\sum}} d^{i}\log n_{i+1}    \\
& & \\
& = & \frac{d^{N+l}(d-1)}{d^{N+l}-1} \left(\underset{\mathbf{w}_{N+l}\in \mathcal{S}(\mathbf{n}_{N+l})}{\max}\left\{\frac{1}{d}\log w_{1}+\frac{1}{d^{2}}\log w_{2}+\cdots +\frac{1}{d^{N+l}}\log w_{N+l}\right\}\right)\\
 & & \\
& = & \frac{d^{N+l}(d-1)}{d^{N+l}-1} \left(\underset{\mathbf{w}_{l}\in \mathcal{S}(\mathbf{y}_{l})}{\max}\left\{\frac{1}{d}\log w_{1}+\frac{1}{d^{2}}\log w_{2}+\cdots +\frac{1}{d^{l}}\log w_{l}\right\}\right),
\end{array}
\end{equation*}
and then
\begin{equation*}
\underset{N\rightarrow\infty}{\lim}h\left(\mathcal{T}_{G}\right)=(d-1)\underset{\mathbf{w}_{l}\in \mathcal{S}(\mathbf{y}_{l})}{\max}\left\{\frac{1}{d}\log w_{1}+\frac{1}{d^{2}}\log w_{2}+\cdots +\frac{1}{d^{l}}\log w_{l}\right\},
\end{equation*}
which means that $\mu^{(d)}_{\mathbf{y}_{l}}$ is an accumulation point of $\mathcal{H}^{(d)}_{irr}$.

In particular, we consider $\mathbf{y}_{l}=\{y_{i}\}_{i=1}^{l}$ with $y_{1}=2^{a}$ and $y_{i}\in\{2^{j}:0\leq j\leq d-1\}$ for $2\leq i\leq l$.
It is clear that
\begin{equation*}
 \mu^{(d)}_{\mathbf{y}_{l}}
 = (d-1)\underset{0\leq m\leq l-1}{\max}\left\{\frac{1}{d}\log y_{m+1}+\cdots +\frac{1}{d^{l-m}}\log y_{l}+\frac{1}{d^{l-m+1}}\log y_{1}+\cdots +\frac{1}{d^{l}}\log y_{m}\right\}.
\end{equation*}
Then, it can be verified that
\begin{equation*}
\begin{array}{rl}
 & \mu^{(d)}_{\mathbf{y}_{l}}= (d-1)\left(\frac{1}{d}\log y_{1}+\frac{1}{d^{2}}\log y_{2}+\cdots +\frac{1}{d^{l}}\log y_{l}\right)\\
 &  \\
 \Leftrightarrow & \left(\frac{1}{d}-\frac{1}{d^{l-m+1}}\right)\log y_{1}\geq
  \underset{i=1}{\overset{l-m}{\sum}}\left(\frac{1}{d^{i}}-\frac{1}{d^{m+i}}\right)\log y_{m+i}-\underset{i=2}{\overset{m}{\sum}}\left(\frac{1}{d^{i}}-\frac{1}{d^{i+l-m}}\right)\log y_{i}\\
  & \\
\Leftrightarrow & \left(\frac{1}{d}-\frac{1}{d^{l-m+1}}\right)\log y_{1}\geq  \frac{d^{2}}{d-1}\left(\frac{1}{d}-\frac{1}{d^{m+1}}\right)
\left(\frac{1}{d}-\frac{1}{d^{l-m+1}}\right)\log 2^{d-1} \\
& \\
 \Leftrightarrow & \log y_{1}\geq \frac{d^{2}}{d-1}\left(\frac{1}{d}-\frac{1}{d^{m+1}}\right)\log 2^{d-1} \\
 & \\
 \Leftrightarrow & a \geq d^{2}\left(\frac{1}{d}-\frac{1}{d^{m+1}}\right).
\end{array}
\end{equation*}
for $1\leq m\leq l-1$.
If $y_{1}\geq 2^{d}$, then $a\geq d \geq d^{2}\left(\frac{1}{d}-\frac{1}{d^{m+1}}\right)$ for all $m\geq 1$, implies
\begin{equation}
\mu^{(d)}_{\mathbf{y}_{l}}= (d-1)\left(\frac{1}{d}\log y_{1}+\frac{1}{d^{2}}\log y_{2}+\cdots +\frac{1}{d^{l}}\log y_{l}\right)
\end{equation}
for all $l\geq 1$. Hence,
\begin{equation*}
\underset{N\rightarrow\infty}{\lim}h\left(\mathcal{T}_{G}\right)=(d-1)\left(\frac{1}{d}\log y_{1}+\frac{1}{d^{2}}\log y_{2}+\cdots +\frac{1}{d^{l}}\log y_{l}\right).
\end{equation*}
Since $\left\{\mu_{\mathbf{y}_{l}}:y_{1}\geq 2^{d} \text{ and } y_{i}\in\{2^{j}:0\leq j\leq d-1\} \text{ for }2\leq i\leq l, l\geq1 \right\}$ is dense in
\begin{equation*}
\left[\frac{(d-1)\log y_{1}}{d},\frac{(d-1)\log y_{1}}{d}+\frac{\log 2}{d}\right).
\end{equation*}
It can be proven immediately that
\begin{equation*}
\frac{(d-1)\log (y_{1}+1)}{d}\leq \frac{(d-1)\log y_{1}}{d}+\frac{\log 2}{d}
\end{equation*}
for $y_{1}\geq 2^{d}$. Therefore, $\mathcal{H}^{(d)}_{irr}$ is dense in
\begin{equation*}
\underset{y_{1}\geq 2^{d}}{\bigcup}\left[\frac{(d-1)\log y_{1}}{d},\frac{(d-1)\log y_{1}}{d}+\frac{\log 2}{d}\right)=[(d-1)\log 2,\infty).
\end{equation*}
 The proof is complete.
\end{proof}

\begin{remark}
\label{remarl:4.2}

It is clear that if $\mu^{(2)}_{\mathbf{y}_{l}}\in\left[\frac{1}{2}\log 2,\frac{1}{2}\log 3\right)$, then $y_{i}\in\{1,2\}$ for all $1\leq i\leq l$, which implies that
\begin{equation*}
\mu^{(2)}_{\mathbf{y}_{l}}=\log 2 \left(\underset{\mathbf{w}_{l}\in \mathcal{S}(\mathbf{x}_{l})}{\max}\left\{\frac{1}{2}x_{1}+\frac{1}{2^{2}}x_{2}+\cdots +\frac{1}{2^{l}}x_{l}\right\}\right),
\end{equation*}
where $\mathbf{x}_{l}=(x_{1},x_{2},\cdots, x_{l})$ with $x_{i}=y_{i}-1\in\{0,1\}$. Hence, we have
\begin{equation*}
\{\mu^{(2)}_{\mathbf{y}_{l}}: \text{ positive integer sequence } \mathbf{y}_{l}=(y_{1},y_{2},\cdots,y_{l}), l\geq 1\}
\end{equation*}
cannot contain any subinterval of $\left[\frac{1}{2}\log 2,\frac{1}{2}\log 3\right)$, and then it is not dense in $\left[\frac{1}{2}\log 2,\frac{1}{2}\log 3\right)\subset\left[\frac{1}{2}\log 2,\log 2\right)$. Therefore, our method in this section fails for proving $\mathcal{H}^{(2)}_{irr}$ is dense in $\left[\frac{1}{2}\log 2,\log2\right)$, it needs further investigation.
 On the other hand, we show numerically that $\mathcal{H}^{(2)}_{irr}$ may have holes in $\left[\frac{1}{2}\log 2,\log2\right)$; see Figure \ref{fig:entropy_distribution} in Section 6. The case for $d\geq 3$ is similar.

\end{remark}

%

\section{Further discussion and open problems}

 Two topics, namely, \textbf{i}. $h(\mathcal{T}_{M})$ for reducible $M$, and \textbf{ii}. denseness of $\mathcal{H}^{(d)}$ and relative denseness of $\mathcal{H}^{(d)}_{irr}$ in $\mathcal{H}^{(d)}$ are discussed in this paper. These two topics are studied in Section 3, 4 and 5, but there are many problems left unanswered.

For Part \textbf{i}., Theorem \ref{theorem:3.2} shows that whether $h(\mathcal{T}_{M(a,b:l)})=\log b$ is highly dependent on the irreducible components $E_{a},E_{b}$ and their connection $R_{l}$. To understand factors affecting $h(\mathcal{T}_{M})$ in more detail,
the more general case for reducible binary matrix $M=M(A,B;R)=\left[\begin{array}{cc}A & R\\ O & B\end{array}\right]$ is considered.
In the following proposition, we take $A\in\{G, G^{*}\}$, $G=\left[\begin{array}{cc} 1 & 1 \\ 1 & 0\end{array}\right]$ and
$G^{*}=\left[\begin{array}{cc} 0 & 1 \\ 1 & 1\end{array}\right]$. Clearly, $G$ and $G'$ has the same eigenvalues. Proposition \ref{proposition:5.1} illustrates that whether $h(\mathcal{T}_{M(A,B;R)})=\max\{h(\mathcal{T}_{A}),h(\mathcal{T}_{B})\}$ is affected by the specific structure of the irreducible components $A$ and $B$, not just their eigenvalues.

\begin{proposition}
\label{proposition:5.1}
Suppose $R=\left[\begin{array}{cc} 1 & 0 \\ 0 & 0 \end{array}\right]$. Then, for $d=2$, we have
\begin{equation*}
h(\mathcal{T}_{M(G,E_{2};R)})>\max\{h(\mathcal{T}_{G}),h(\mathcal{T}_{E_{2}})\}=\log 2,
\end{equation*}
and
\begin{equation*}
h(\mathcal{T}_{M(G^{*},E_{2};R)})=\max\{h(\mathcal{T}_{G^{*}}),h(\mathcal{T}_{E_{2}})\}=\log 2.
\end{equation*}
\end{proposition}

\begin{proof}
First, $M=M(G,E_{2};R)$ is considered.
From \cite{petersen2018tree}, we have
\begin{equation*}
 \left\{
  \begin{array}{l}
 p_{1}(n+1)=\left(p_{1}(n)+p_{2}(n)+p_{3}(n)\right)^{2},\\
 p_{2}(n+1)=\left(p_{1}(n)\right)^{2},\\
 p_{3}(n+1)= p_{4}(n+1)=\left(p_{3}(n)+p_{4}(n)\right)^{2}=2^{2^{n+2}-2},
 \end{array}
\right.
\end{equation*}
with $ p_{i}(0)=1$, $1\leq i\leq 4$. Let $\chi_{n}=p_{1}(n)/p_{3}(n)$, $n\geq 0$. In particular, $\chi_{0}=1$ and $\chi_{1}=9/4>1$. Clearly,
\begin{equation*}
\chi_{n+1}=\left(\frac{1}{2}\chi_{n}+\frac{p_{2}(n)}{2p_{3}(n)}+\frac{1}{2}\right)^{2}\geq \left(\frac{1}{2}\chi_{n}+\frac{1}{2}\right)^{2}
\geq \frac{1}{4}\chi_{n}^{2},
\end{equation*}
for $n\geq 0$.

Let $f(x)=\left(\frac{1}{2}x+\frac{1}{2}\right)^{2}$. It can be checked that $f(x)$ is increasing for $x\geq 1$, and there
exists unique intersection point of $y=f(x)$ and $y=x$ at $x=1$.  Since $\chi_{n+1}\geq f(\chi_{n})$ with initial term $\chi_{1}>1$, $\{\chi_{n}\}$ approaches $\infty$ as $n$ tends to $\infty$. Similar to the proof of Theorem \ref{theorem:3.0.2}, from $\chi_{n+1}\geq \frac{1}{4}\chi_{n}^{2}$, we have for $n\geq 1$,
\begin{equation*}
\limsup\limits_{m \to \infty}\frac{\log \chi_{n+m}}{2^{n+m+1}-1} \geq 2^{-n-1} (\log \chi_{n}-2\log2),
\end{equation*}
which yields $h(\mathcal{T}_{M(G,E_{2};R)})> \log 2$.

On the other hand, consider $M=M(G^{*},E_{2};R)$. Since
$M(G^{*},E_{2};R)$ has the same row sum $2$, we can obtain the estimation $2^{2^{n+1}-1}\leq p_{M(G^{*})}(n)\leq 4\cdot2^{2^{n+1}-2}$ for $n\geq 0$. Therefore, $h(\mathcal{T}_{M(G^{*},E_{2};R)})= \log 2$. The proof is complete.

\end{proof}

\begin{problem}
In Section 3, we reveal that $A$, $B$ and $R$ play important roles on the problem whether the inequality $h(\mathcal{T}_{M(A,B;R)}) = \max \{h(\mathcal{T}_{A}),h(\mathcal{T}_{B})\}$ holds. However, the set of $(A,B;R)$ discussed in Section 3 seems to be restricted. Such problem for general $(A,B;R)$ remains.

\end{problem}

For one-dimensional shifts of finite type, it is known that the computation of the value $h(X_{M(A,B;R)})$ is independent of $R$. Theorem \ref{theorem:3.2} reveals that $R$ has a dramatic influence on the computation of $h(\mathcal{T}_{M(A,B;R)})$. Table 5 shows numerically that for $d=2$ and $A=B=G$, $h(\mathcal{T}_{M(G,G;R}))>h(\mathcal{T}_{G})$ for non-zero matrix $R$, but $h(\mathcal{T}_{M(G,G;R)})$ is different for all $R$. It illustrates that $h(\mathcal{T}_{M})$ for general reducible $M$ may be affected by all connection $R$'s among irreducible components.

\begin{center}
\begin{tabular}{|c|c|c|c|c|c|}
\hline
$R=[r_{1,1},r_{1,2};r_{2,1},r_{2,2}]$ & $h(\mathcal{T}_{M(G,G;R)})$ & $h(\mathcal{T}_{G})$\\
\hline
[0,0;0,0]& 0.20898764 & 0.20898764
 \\
\hline
[0,0;0,1]& 0.26939373
 & 0.20898764
 \\
\hline[0,0;1,0]& 0.28511043
 & 0.20898764
 \\
\hline[0,0;1,1]& 0.323928413
 & 0.20898764
 \\
\hline[0,1;0,0]& 0.310989658
 & 0.20898764
 \\
\hline[0,1;0,1]& 0.332599455
 & 0.20898764
 \\
\hline[0,1;1,0]& 0.341668176
 & 0.20898764
 \\
\hline[0,1;1,1]& 0.365964234
 & 0.20898764
 \\
\hline[1,0;0,0]& 0.330387956
 & 0.20898764
 \\
\hline[1,0;0,1]& 0.349104083
 & 0.20898764
 \\
\hline[1,0;1,0]& 0.355789035
 & 0.20898764
 \\
\hline[1,0;1,1]& 0.378081254
 & 0.20898764
 \\
\hline[1,1;0,0]& 0.381226587
 & 0.20898764
 \\
\hline[1,1;0,1]& 0.393079449
 & 0.20898764
 \\
\hline[1,1;1,0]& 0.397923765
 & 0.20898764
 \\
\hline[1,1;1,1]& 0.413311852
 & 0.20898764
 \\
 \hline
\end{tabular}
\end{center}
\begin{equation*}
\text{Table 5.}
\end{equation*}
\begin{figure}[H]
    \centering
    \includegraphics[trim=180 0 0 0,scale=0.43]{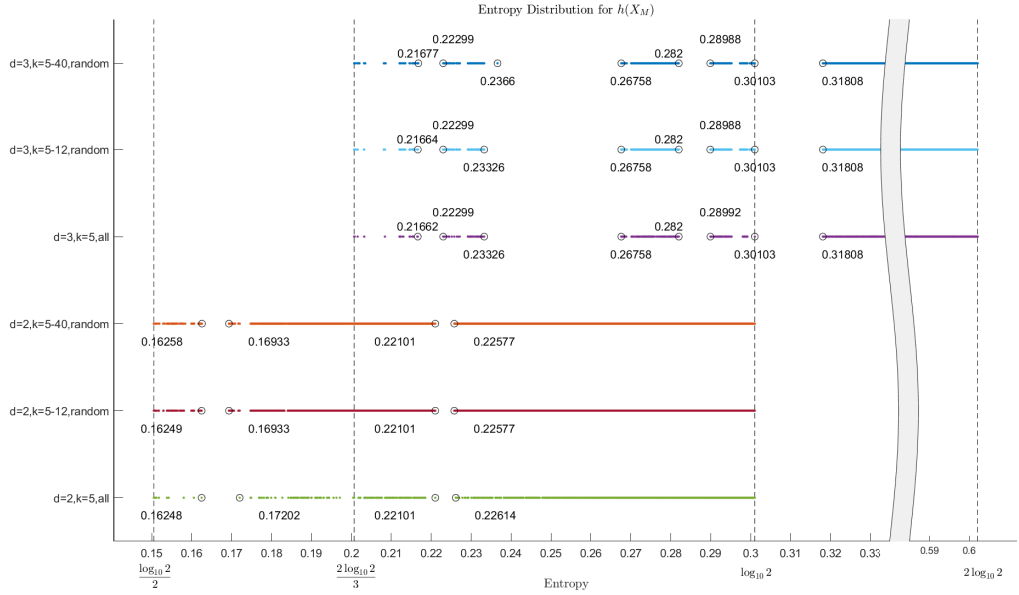}
    \caption{Distribution of entropies $\mathcal{H}^{(d)}$ for $d=2$ and $d=3$ ($\log$ is computed with base $10$.)}
\label{fig:entropy_distribution}
\end{figure}

For part \textbf{ii}., since Theorem \ref{theorem:2.2.1} shows that for $d\geq 2$, $\mathcal{H}^{(d)}_{irr}$ is dense in $[(d-1)\log 2,\infty)$ and the authors \cite{ban2020topological} prove $\mathcal{H}^{(d)} \bigcap\left(0,\frac{d-1}{d}\log2\right)=\emptyset$.
A straightforward problem is to determine whether $\mathcal{H}^{(d)}_{irr}$ or $\mathcal{H}^{(d)}$ is dense in $\left(\frac{d-1}{d}\log 2, (d-1)\log 2 \right)$. Our method in Section 4 can not work in this circumstances, and further study is needed to solve this problem.
The following figure illustrates numerically that $\mathcal{H}^{(2)}$ may have two holes in the intervals $(0.16579,0.16933)$ and $(0.22219,0.22577)$, and $\mathcal{H}^{(3)}$ also may have four holes in the intervals $(0.21677,0.22299)$, $(0.2366,0.26758)$, $(0.282, 0.28988)$ and $(0.30103,0.31808)$. Therefore, it is possible that $\mathcal{H}^{(d)}$ is not dense in $\left(\frac{d-1}{d}\log 2, (d-1)\log 2 \right)$.

Since it can happen that $h(\mathcal{T}_{M})>\underset{1\leq i\leq q}{\max}\{h(\mathcal{T}_{M_{i}})\}$, the following problem is raised naturally.
\begin{problem}
Is $\mathcal{H}^{(d)}_{irr}=\mathcal{H}^{(d)}$?
\end{problem}
For one-dimensional SFTs, $\left\{h(X_{M}):M \text{ is binary}\right\}$ is equal to the set of all logarithms of Perron numbers in $[1,\infty)$; for higher-dimensional SFTs, Hochman and Meyerovitch \cite{HM-AoM2010a} prove that a non-negative number can be an entropy if and only if it is \emph{right recursively enumerable} (the detailed definition is omitted here, see \cite{HM-AoM2010a}).

\begin{problem}
Is there any classification result on the entropies of hom Markov tree-shifts? In particular, does the class $\mathcal{H}^{(d)}$ contains all logarithms of Perron numbers or all right recursively enumerable numbers in its dense intervals?
\end{problem}

\bibliographystyle{amsplain_abbrv}
\bibliography{ban}

\end{document}